\documentclass[12pt, reqno]{amsart}
\usepackage{amssymb,a4}
\usepackage{amsmath,amsfonts,amssymb,amsthm,amscd}
\usepackage{mathrsfs}

\usepackage{color}
\usepackage{colordvi}

\newcommand{\on}{\operatorname}

\newcommand{\mbf}{\mathbf}

\def\C{{\mathbb C}}

\def\Z{{\mathbb Z}}

\def\N{{\mathbb N}}
\def\1{{\bf 1}}

\def \tr{{\rm tr}}

\def \End{{\rm End}}

\def \Hom{{\rm Hom}}

\def \<{\langle}
\def \>{\rangle}
\def \w{\omega}

\def \h{\mathfrak{h}}

\def \w{\omega}
\numberwithin{equation}{section}

\newtheorem{theorem}{Theorem}[section]
\newtheorem{prop}[theorem]{Proposition}
\newtheorem{lem}[theorem]{Lemma}
\newtheorem{cor}[theorem]{Corollary}
\theoremstyle{definition}
\newtheorem{defn}[theorem]{Definition}

\newtheorem{remark}[theorem]{Remark}
\newtheorem{proofof}[theorem]{}

\begin{document}
\title[Heisenberg Vertex Operator Algebras]{The varieties of Heisenberg vertex operator algebras}

\author{ Yanjun Chu}
\address[Chu]{School of Mathematics and Statistics,  Henan University, Kaifeng, 475004, China}
\email{chuyj@henu.edu.cn}

\author{Zongzhu Lin}
\address[Lin]{ Department of Mathematics, Kansas State University, Manhattan, KS 66506, USA}
\email{zlin@math.ksu.edu}
\begin{abstract}
For a vertex operator algebra $V$ with conformal vector $\omega$, we consider a class of vertex operator subalgebras
and their conformal vectors. They are called semi-conformal vertex operator subalgebras and semi-conformal vectors of  $(V,\omega)$, respectively, and  were used to study duality theory of vertex operator algebras via coset constructions. Using these objects attached to $(V,\omega)$,  we shall understand the structure of the vertex operator algebra $(V,\omega)$. At first, we define the set $\on{Sc}(V,\omega)$ of semi-conformal vectors of $V$,  then we prove that $\on{Sc}(V,\omega)$ is an affine algebraic variety with a partial ordering and an involution map. Corresponding to each semi-conformal vector, there is a unique maximal semi-conformal vertex operator subalgebra containing it. The properties of these subalgebras are invariants of vertex operator algebras.  As an example, we describe the corresponding varieties of  semi-conformal vectors for Heisenberg vertex operator algebras. As an application, we give two characterizations  of Heisenberg vertex operator algebras using the properties of  these varieties.
\end{abstract}

\maketitle
\section{Introduction}

\subsection{}
 A vertex operator algebra is an algebraic structure that plays an important role in conformal field theory and string theory. In addition to physical applications, vertex operator algebras have been proven useful in purely mathematical contexts such as in representation theory of affine Lie algebras  and the geometric Langlands programs. This paper is intended to understand properties of vertex operator algebras in terms of a certain geometric objects attached to them.
\subsection{}
In  \cite{JL2}, Jiang and the second author  used a class of subalgebras for a vertex operator algebra $V$ to study level-rank duality in vertex operator algebra theory. For a vertex operator algebra $V$ with  conformal vector $\omega$,  a weight-two vector $\omega'$ is said to be a semi-conformal vector if $\omega'$ is the conformal vector of a vertex operator subalgebra $U$ (with possibly different conformal vectors as in \cite[Section 3.11.6]{LL}) such that $\omega_n|_U=\omega'_n|_U$ for all $n\geq 0$ and such a vertex operator subalgebra  $U$ is called a semi-conformal subalgebra of $V$. Let  $\on{Sc}(V,\omega)$ be the set of semi-conformal vectors of $V$.
\begin{theorem}
 For a vertex operator algebra $(V,\omega)$, the set $\on{Sc}(V,\omega)$  is an affine algebraic variety over $\mathbb{C}$.
\end{theorem}
The purpose of this paper is to understand vertex operator algebras   using the  properties of the algebraic variety  $\on{Sc}(V,\omega)$.

Denote the set of semi-conformal subalgebras of $(V,\omega)$ by $\on{ScAlg}(V,\omega)$.
Obviously, we have a surjective map
$$\begin{array}{llll}
\pi:\on{ScAlg}(V,\omega)\longrightarrow \on{Sc}(V,\omega)\\
\hspace{1.5cm}(U,\omega')\longmapsto \omega'.
\end{array}
$$
Since different semi-conformal subalgebras of $(V, \omega)$ can share the same conformal vector $\omega'$, then the fiber $\pi^{-1}(\omega')$ for each $\omega'\in \on{Sc}(V,\omega)$ is the set of semi-conformal subalgebras with the same conformal vector $\omega'$ in $V$. In the fiber $\pi^{-1}(\omega')$, there exists a unique maximal one $U(\omega')$, which is the maximal conformal extension in $V$ of semi-conformal subalgebras with the
conformal vector $\omega'$ and can be realized as the double commutant $C_{V}(C_{V}(<\omega'>))$ of $<\omega'>$ in $V$, where $<\omega'>$ is a Virasoro vertex operator algebra generated by $\omega'$.  In fact,
semi-conformal vectors classify semi-conformal  subalgebras of $V$ up to conformal extensions. Thus, by studying the fiber $\pi^{-1}(\omega')$  for each $\omega'\in \on{Sc}(V,\omega)$, we can classify the set $\on{ScAlg}(V,\omega)$.
\subsection{}
For any two semi-conformal vectors $\omega',\omega''$ of $(V,\omega)$, we  can  define  a partial ordering $\omega'\preceq\omega''$ if $U(\omega')\subseteq U(\omega'')$. On the other hand, according to the commutant  properties of
 a vertex operator algebra, we have an operator $C_V$ on $\on{ScAlg}(V,\omega)$
 as follows
 $$
 \begin{array}{lllll}
 C_V:\on{ScAlg}(V,\omega)\longrightarrow \on{ScAlg}(V,\omega)\\
 \hspace{1.85cm}(U,\omega')\longmapsto C_{V}(U).
 \end{array}
 $$
 Then $C_V$ induces an involution operator $\omega-$ on $\on{Sc}(V,\omega)$ as follows
 $$
 \begin{array}{lllll}
 \omega- :\on{Sc}(V,\omega)\longrightarrow \on{Sc}(V,\omega)\\
 \hspace{2.0cm}\omega'\longmapsto \omega-\omega'.
 \end{array}
 $$
In fact,  the following  diagram commutes
 $$\begin{CD}
 \on{ScAlg}(V,\omega)@>C_V>> \on{ScAlg}(V,\omega)\\
 @VV\pi V      @VV\pi V\\
 \on{Sc}(V,\omega)@>\omega->> \on{Sc}(V,\omega).
 \end{CD}
 $$
These operators depend on coset constructions of a vertex operator algebra and reflect certain duality constructions in representations of a vertex operator algebra similar to duality pairs of Lie groups (\cite[Section 5]{JL2}).

We also note that the group $G=\on{Aut}(V, \omega)$ (the automorphism group of  conformal vector being  preserved) acts on both $\on{ScAlg}(V, \omega)$ and $\on{Sc}(V, \omega)$ and the map $\pi$ is $G$-equivariant.
\subsection{}
 Based on above approaches of studying vertex operator algebras, we take Heisenberg vertex operator algebras as examples. Let $\h$ be a $d$-dimensional orthogonal space, i.e, $\h$ is a vector space equipped with a nondegenerate symmetric bilinear form $\langle \cdot, \cdot\rangle$.  Then $\h$ generates a Heisenberg vertex operator algebra $V_{\widehat{\h}}(1,0)$ of level $1$ with a  fixed conformal vector $\omega=\omega_{\Lambda}$ with $\Lambda=0$ (see Section~\ref{sec3.1}).  Let $\h'\subset \h$ be a subspace of $\h$. If the bilinear form  $\langle\cdot, \cdot\rangle|_{\h'}$ on $\h$ is still nondegenerate, we say $\h'$ is a {\em regular subspace} of $\h$. Set $$\on{Reg}(\h)=\{\h'\;|\;\h'\; \text{is a regular subspace of } \; \h\}.$$
For each $\omega'\in \on{Sc}(V_{\widehat{\h}}(1,0),\omega)$, we will construct   a unique symmetric idempotent matrix $A_{\omega'}$, which
 defines a linear transformation $\mathcal{A}_{\omega'}: \h\rightarrow \h $ with respect to the fixed orthonormal basis of $\h$ and thus corresponds to two regular subspaces $\on{Im}\mathcal{A}_{\omega'}$ and $ \on{Ker}\mathcal{A}_{\omega'}$ of $\h$ giving an orthogonal decomposition $\on{Im}\mathcal{A}_{\omega'}\oplus \on{Ker}\mathcal{A}_{\omega'} =\h$.  
 
 We note that that group $\on{Aut}(V_{\widehat{\h}}(1,0))=\on{O}(\h)$  acts on $\on{Reg}(\h)$ naturally. We have the following description of the variety $\on{Sc}(V_{\widehat{\h}}(1,0),\omega)$. Also $\on{Reg}(\h)$ is a partially ordered set under the subspace inclusion relation.

 \begin{theorem} \label{thm1.2}
 \begin{itemize}
\item [1)] The map $\rho: \omega'\mapsto \on{Im}(\mathcal{A}_{\omega'})$ is an ordering preserving $\on{Aut}(V_{\widehat{\h}}(1,0), \omega)$-equivariant bijection  form $\on{Sc}(V_{\widehat{\h}}(1,0),\omega)$ to $\on{Reg}(\h)$;

\item [2)]
  $\on{Sc}(V_{\widehat{\h}}(1,0),\omega) $ has exactly
 $d+1$ orbits under the group $\on{Aut}(V_{\widehat{\h}}(1,0),\omega)$-action and  each $0\leq i \leq d$ corresponds to the orbit
 $$\on{Sc}(V_{\widehat{\h}}(1,0),\omega)_i=\{\h'\subset \h| \h'~is~ a ~regular~ subspace~ of ~\h~ with~\on{dim}\h'=i\};$$

\item [3)] There exists a longest chain in $\on{Sc}(V_{\widehat{\h}}(1,0),\omega)$
    such that the length of this chain  equals to $d$:
 there exist $\omega^1,\cdots,\omega^{d-1}\in \on{Sc}(V_{\widehat{\h}}(1,0),\omega)$ such that
 $$0=\omega^0\prec \omega^1\prec\cdots\prec\omega^{d-1}\prec \omega^d=\omega.$$
 \end{itemize}
 \end{theorem}
 For each $ \omega' \in \on{Sc}(V_{\widehat{\h}}(1,0),\omega)$, each of the  abelian Lie algebras $\on{Im}\mathcal{A}_{\omega'}$ and $ \on{Ker}\mathcal{A}_{\omega'}$    generates a Heisenberg vertex operator subalgebra in $V_{\widehat{\h}}(1,0)$. In fact, they are both semi-conformal subalgebras of $(V_{\widehat{\h}}(1,0),\omega)$ and can be both realized as commutant subalgebras of $(V_{\widehat{\h}}(1,0),\omega)$.


\begin{theorem}\label{thm1.3}
For each  $\omega'\in \on{Sc}(V_{\widehat{\h}}(1,0),\omega)$, the following assertions hold. 
\begin{itemize}
\item[1)] $\on{Im}\mathcal{A}_{\omega'}$ generates a Heisenberg vertex operator algebra $$ V_{\widehat{\on{Im}\mathcal{A}_{\omega'}}}(1,0)=C_{V_{\widehat{\h}}(1,0)}(<\omega-\omega'>)$$
and $\on{Ker}\mathcal{A}_{\omega'}$ generates a Heisenberg vertex operator algebra $$V_{\widehat{\on{Ker}\mathcal{A}_{\omega'}}}(1,0)=C_{V_{\widehat{\h}}(1,0)}(<\omega'>);$$

\item[2)]
$C_{V_{\widehat{\h}}(1,0)}(V_{\widehat{\on{Ker}\mathcal{A}_{\omega'}}}(1,0))=V_{\widehat{\on{Im}\mathcal{A}_{\omega'}}}(1,0);\;  C_{V_{\widehat{\h}}(1,0)}(V_{\widehat{\on{Im}\mathcal{A}_{\omega'}}}(1,0)))=V_{\widehat{\on{Ker}\mathcal{A}_{\omega'}}}(1,0);$

\item[3)] $ V_{\widehat{\h}}(1,0)\cong C_{V_{\widehat{\h}}(1,0)}(<\omega'>)\otimes C_{V_{\widehat{\h}}(1,0)}(C_{V_{\widehat{\h}}(1,0)}(<\omega'>)).$
\end{itemize}
\end{theorem}
\subsection{}\label{sec1.5}
Based on  above results (See Theorem\ref{thm1.2} and Theorem~\ref{thm1.3}),
we can give two characterizations of the Heisenberg vertex operator algebra $(V_{\widehat{\h}}(1,0),\omega)$.

  In this paper, a vertex operator algebra $(V,\omega)$ will satisfy the following conditions:
(1) $V$ is a {\em simple CFT type} vertex operator algebra (i.e., $V$ is $\N$-graded and $V_0=\C \mbf{1}$); (2) The symmetric bilinear form $<u,v>=u_1v$ for $u,v\in V_1$ is  nondegenerate. For convenience, we call  such a vertex operator algebra $(V,\omega)$   {\em nondegenerate simple CFT type. } We note that for any vertex operator algebra $(V, \omega)$ and any $\omega' \in \on{Sc}(V,\omega)$, one has
$C_{V}(C_{V}(<\omega'>))\otimes C_{V}(<\omega'>)\subseteq V$ as a conformal vertex operator subalgebra.

\begin{theorem} \label{thm1.4}
Let $(V,\omega)$ be a nondegenerate simple CFT type vertex operator algebra generated by $V_1$. Assume that $ L(1)V_1=0$. 
If for
each $\omega'\in \on{Sc}(V,\omega)$, there are
\begin{equation} \label{e5.5}V\cong C_{V}(C_{V}(<\omega'>))\otimes C_{V}(<\omega'>)\end{equation}
then $(V,\omega)$ is isomorphic to the Heisenberg vertex operator algebra $(V_{\hat{\h}}(1,0), \omega_{\Lambda})$ with $\h=V_1$ and $ \Lambda=0$. 
\end{theorem}
 For a semi-conformal vector $\omega'$  of a vertex operator algebra $(V,\omega)$, we can define the height and depth of $\omega'$ in $\on{Sc}(V, \omega)$ analogous to those concepts of prime ideals in a commutative ring. This is also one of the motivations of studying the set of all semi-conformal vectors.  Considering the maximal length of  chains of semi-conformal vectors in $V$,  we can give another characterization of  Heisenberg vertex operator algebras.

\begin{theorem} \label{thm1.5}
Let $(V,\omega)$ be a nondegenerate simple CFT type vertex operator algebra  generated by $V_1$. Assume $\on{\dim}V_1=d$ and $ L(1)V_1=0$.  If there exists a chain $0=\omega^0\prec\omega^1\prec\cdots \prec \omega^{d-1}\prec \omega^{d}=\omega$
in $\on{Sc}(V,\omega)$ such that $\on{dim}C_{V}(C_{V}(<\omega^{i}-\omega^{i-1}>))_1\neq 0, ~for~i=1,\cdots, d$, then  $V$ is isomorphic to the Heisenberg vertex operator algebra $(V_{\widehat{\h}}(1,0), \omega_{\Lambda})$ with  $\h=V_1$ and  $ \Lambda=0$. 
\end{theorem}
Characterizing the Monster Moonshine module by its properties is one of the focuses in conformal field theory. The FLM conjecture in \cite{FLM} is one of the conjectural characterizations of  the Monster Moonshine module.   
\subsection{} By \cite[Lemma 5.1]{JL2}, a vertex subalgebra $U$ of $V$ cannot contain two different conformal vectors $\omega'$ and $\omega''$ such that both $(U, \omega')$ and  $(U, \omega'')$ are semi-conformal vertex operator subalgebras of $ (V, \omega)$. We remark that on a vertex algebra, there can be many different conformal vectors to make it become non-isomorphic vertex operator algebras, even thought they have the same conformal gradations. The result of  \cite[Lemma 5.1]{JL2} says that no one is a semi-conformal with respect to another. \cite[Example 2.5.9]{BF} provides a large number of such examples on the Heisenberg vertex operator algebras.  Thus the map from $\on{ScAlg}(V, \omega)$ to the set of all vertex subalgebras of $V$, forgetting the conformal structure,  is injective.  The conformal vector in $ V$ uniquely determines the semi-conformal structure on a vertex subalgebra of $V$ if there is any.

One of the main  motivations of this work is to investigate the conformal structure on a vertex subalgebra of a vertex operator algebra.  In conformal field theory, the conformal vector completely determines the conformal structure (the module structure for the Virasoro Lie algebra).  In mathematical physics, a vertex operator algebra has been investigated extensively as a Virasoro module (see \cite{DMZ,D,KL,DLM,LY,M,LS,S,L,Sh}) by virtue of conformal vector.

 \subsection{}
This is the first paper of a series of papers investigating the geometric properties of $\on{Sc}(V, \omega)$ and structures of $(V, \omega)$. In an upcoming paper we will consider the cases when $(V, \omega)$ is of affine types and lattice types. As the Heisenberg case indicated here, the variety $\on{Sc}(V, \omega)$ is closely related to the affine Lie algebra with a fixed level as well as to the lattice structure.

 This paper is organized as follow: In Section 2, we review semi-conformal vectors (subalgebras) of a vertex operator algebra according to \cite{JL2}. Then we study  properties of  the variety  of semi-conformal vectors of a vertex operator algebra. In Section 3, we shall describe the variety of semi-conformal vectors of the  Heisenberg vertex operator algebra $(V_{\widehat{\h}}(1,0),\omega_{\Lambda})$ with $\Lambda=0$. In Section 4, we study  maximal semi-conformal subalgebras  determined by semi-conformal vectors of  $(V_{\widehat{\h}}(1,0),\omega_{\Lambda})$ with $\Lambda=0$. In Section 5, we shall give two characterizations of Heisenberg vertex operator algebras in terms of  semi-conformal vectors and the corresponding subalgebras. In Section 6, we summarize 
the conclusion and give further discussion.

{\em Acknowledgement:} This work started when the first author was visiting Kansas State University from September 2013 to September 2014. He thanks the support by Kansas State University and its hospitality.  The first author also thanks China Scholarship Council for their financial supports. The second author thanks C. Jiang for many insightful discussions. This work was motivated from the joint work with her. The second author also thanks Henan University for the hospitality during his visit in the summer of 2015, during which this work was carried out. The authors want to thank the referees for carefully reading the paper and for providing constructive comments.  

\section{ Semi-conformal vectors and semi-conformal subalgebras of a vertex operator algebra}
\setcounter{equation}{0}

For basic notions and results associated with vertex operator algebras, one is
referred  to the books \cite{FLM,LL,FHL,BF}.  We will use $(V, Y, 1)$ to denote a vertex algebra and $(V, Y, 1, \omega)$ for a vertex operator algebra. When we deal with several different vertex algebras, we will use $ Y^V$, $1^{V}$, and $\omega^V$ to indicate the dependence of the vertex algebra or vertex operator algebra $V$.  For example $Y^V(\omega^V, z)=\sum_{n\in \Z} L^V(n)z^{-n-2}$.  To emphasize the presence of the conformal vector $\omega^V$, we will simply write $(V, \omega^V)$ for a vertex operator algebra and $V$ simply for a vertex algebra (with $Y^V$ and $1^V$ understood). We refer \cite{BF} for the concept of vertex algebras. Vertex algebras need not be graded, while a vertex operator algebra $(V,\omega)$ is always $ \Z$-graded by the $L^V(0)$-eigenspaces $V_n$ with integer eigenvalues $ n \in \Z$. We assume that each $ V_n$ is finite dimensional over $\C$ and $ V_n=0$ if $ n<<0$.   

    In this section, we shall first review semi-conformal vectors (subalgebras) of a vertex operator algebra (\cite{JL2}).

For two given vertex algebras $(V, Y^V , 1^V )$ and $(W, Y^W, 1^W),$ we recall from  \cite[Section 3.9]{LL} that
a homomorphism $f : V\rightarrow W$ of vertex algebras satisfies
$$f(Y^V (u, z)v) = Y^W(f(u),z)f(v),~~ \forall u, v \in V;~f(1^V ) = 1^W.$$
Let $V$ and $W$ be two vertex operator algebras with conformal vectors
$\omega^V$ and $\omega^W$, respectively. Then  $f$  is called {\em conformal} if $f(\omega^V ) = \omega^W$, which is equivalent to
$f \circ L^V (n) = L^W(n) \circ f, ~\mbox{for~all}~ n \in\Z.$ We say $f$  is {\em semi-conformal}
if $f \circ L^V (n) = L^W(n) \circ f,$ for all $n\geq 0$.

We remark that, for any vertex algebra homomorphism $f$ between two vertex operator algebras $(V, Y^V, 1^V, \omega^V)$ and $ (W, Y^W, 1^V, \omega^W)$, one always has $f \circ L^V (-1) = L^W(-1) \circ f$. Also $f$ is conformal if and only if $f \circ L^V (n) = L^W(n) \circ f,$ for all $n\geq -2$.  Thus a  semi-conformal  homomorphism $f$ is conformal if and only if $f \circ L^V (-2) = L^W(-2) \circ f $. 

Let $(V,\omega^V)$ be a vertex subalgebra of $(W,\omega^W)$ and
the map $f: V \rightarrow W$ is the inclusion, we say $V$ is a
conformal subalgebra of $W$ if $f$ is
conformal ($V$ has the same conformal vector with $W$).
If $f$ is semi-conformal, then $V$ is called a semi-conformal
subalgebra of $W$ and $\omega^V$ of $V$ is called
a semi-conformal vector of $W$.

For a vertex operator algebra $(W,\omega^W)$, we define
$$\begin{array}{lllll}
&\on{ScAlg}(W,\omega^W)=\{(V,\omega^V)\; |\; (V,\omega^V)~ \text{is~a~semi-conformal~subalgebra~of~}(W,\omega^W)\};\\
&\on{Sc}(W,\omega^W)=\{\omega'\in W\; |\;\omega'~\text{is~a~semi-conformal~vector~of}~(W,\omega^W)\};\\
&\overline{\on{S}(W,\omega^W)}=\{(V,\omega')\in \on{ScAlg}(W,\omega^W)|C_{W}(C_{W}(V))=V\}.
\end{array}
$$
Here $ C_{W}(V)$ is the centralizer of the vertex subalgebra $V$ in $W$ defined in \cite[Section 3.11]{LL}. It is also called commutant by others. 

It follows from the definition that there is a surjective map $ \on{ScAlg}(W, \omega^W)\rightarrow \on{Sc}(W,\omega^W)$ by $(V, \omega^V)\mapsto \omega^V$  since  the vertex subalgebra $ <\omega'>$ generated by $ \omega'$ and $1^W$ is automatically a semi-conformal subalgebra.  There is also a surjective map $ \on{ScAlg}(W, \omega^W)\rightarrow \overline{\on{S}(W,\omega^W)}$ defined by $(V, \omega^V)\mapsto (C_W(C_W(V), \omega^V)$.

\begin{prop}\label{p2.1}
The restriction of the map $\on{ScAlg}(W,\omega^W)\rightarrow \on{Sc}(W, \omega^W)$ to the set $\overline{\on{S}(W,\omega^W)}$  is a bijection.
\end{prop}
\begin{proof} The map $\omega'\mapsto C_{W}(C_W(<\omega'>))$ is the inverse map  $\on{Sc}(W, \omega^W)\rightarrow \overline{\on{S}(W,\omega^W)}$.
\end{proof}

In a special case that  $(W,\omega^W)$ is a $\N$-graded vertex operator algebra with $W=\coprod\limits_{n\in \N} W_n$ and $W_0=\C\mathbf{1}$, the condition for a  vertex operator  subalgebra to be semi-conformal is much simpler. In this case,  for any vertex subalgebra $V$ of $W$, if $(V,\omega^{V})$ is a vertex operator algebra for  some $\omega^V\in V$, then
$(V,\omega^{V})$ is a semi-conformal subalgebra of $W$ if and only if
$\omega^{V}\in W_2\bigcap \mbox{Ker} L^W(1)$(see \cite[Theorem 3.11.12]{LL}).
Moreover, if $(V,\omega^{V})$ is a semi-conformal subalgebra of $(W,\omega^{W})$,
then $(C_{W}(V),\omega^{C}=\omega^W-\omega^V)$ is also a semi-conformal
subalgebra of $(W,\omega^{W})$, that is, if $\omega^{V}\in \on{Sc}(W,\omega^W)$, then $\omega^C=\omega^{W}-\omega^{V}\in \on{Sc}(W,\omega^W)$. Let $(V,\omega^V), (U,\omega^U)$ be two semi-conformal subalgebras of $(W,\omega^W)$.
If ~$\omega^V=\omega^U$ and $V\subset U$, then we say $(U,\omega^{U})$ is a conformal extension of $(V,\omega^{V})$ in $(W,\omega^{W})$. Moreover, $(V,\omega^{V})$ has a unique maximal conformal extension
$(C_{W}(C_{W}(V)),\omega^{V})$ in $(W,\omega^{W})$ in the sense that
if $(V,\omega^V)\subset (U, \omega^V)$, then $(U, \omega^V)\subset  (C_{W}(C_{W}(V)),\omega^{V})$( see \cite[Corollary 3.11.14]{LL}).  A semi-conformal subalgebra $(U, \omega^U)$ of $W$ is called {\em conformally closed} if $ C_V(C_V(U))=U$ (see \cite{JL2}). So  the set $\overline{\on{S}(W,\omega^W)}$ consists of all conformally closed
semi-conformal subalgebras of $(W,\omega^W)$.




Let  $(W,\omega^W)$ be a general $\Z$-graded vertex operator algebra. For a vector $\omega'\in W$, we write  $Y(\omega', z)=L'(z)=\sum_{n\in \Z}L'(n)z^{-n-2}$. Then we can use the following equations to characterize the set of semi-conformal vectors of $(W,\omega^W)$.
\begin{prop}\label{p2.2}
A vector $\omega'\in W_2$ is a semi-conformal vector of $(W,\omega^W)$ if and only if
it satisfies
\begin{align}\label{2.1}\left\{\begin{array}{llll}
L'(0)\omega'=L^W(0)\omega'=2\omega';\\
L'(1)\omega'=L^W(1)\omega'=0;\\
L'(2)\omega'=L^W(2)\omega'=\frac{c'}{2}\mbf{1}, \quad \text{for some } c'\in \C;\\
L'(-1)\omega'=L^W(-1)\omega';\\
L'(n)\omega'=L^W(n)\omega'=0, n\geq 3.
\end{array}\right.
\end{align}

\end{prop}
\begin{proof} If $\omega'$ is a semi-conformal vector, then $\omega'$ generates a Virasoro vertex operator algbera $<\omega'>$ in $W$ in sense of \cite[Proposition 3.9.3]{LL}, which is also a semi-conformal
subalgebra of $W$. Note that
$\omega'$ is the conformal vector of $<\omega'>$. By \cite[Corollary 3.11.9 and Corollary 3.11.10 ]{LL}, we have
\begin{equation}\label{2.2}[L'(m),L'(n)]=(m-n)L'(m+n)+\frac{m^3-m}{12}\delta_{m+n,0}c' \end{equation}
 acting on $W$ for some complex number $c'$, that is, there are
$$\left\{\begin{array}{llll}
L'(0)\omega'=2\omega';\\
L'(1)\omega'=0;\\
L'(2)\omega'=\frac{c'}{2}  1;\\
L'(-1)\omega'=L(-1)\omega';\\
L'(n)\omega'=0, n\geq 3.
\end{array}\right.
$$
Since $\omega'$ is a semi-conformal vector of $W$, we have
\begin{equation*}\left\{\begin{array}{llll}
L'(0)\omega'=L^W(0)\omega'=2\omega';\\
L'(1)\omega'=L^W(1)\omega'=0;\\
L'(2)\omega'=L^W(2)\omega'=\frac{c'}{2} \mbf{1};\\
L'(-1)\omega'=L^W(-1)\omega';\\
L'(n)\omega'=L^W(n)\omega'=0, n\geq 3.
\end{array}\right.
\end{equation*}

Conversely, if the conditions \eqref{2.1} hold,
then  there are  $L'(n)=L^W(n)$ on  Virasoro vertex operator algebra  $<\omega'>$ for all $n\geq -1$. Hence $\omega'$
 is a semi-conformal vector of $(W,\omega^W)$.
\end{proof}
\begin{remark}
Let $(W,\omega^W)$ be a $\N$-graded vertex operator algebra with $W_0=\C \mathbf{1}$. According to
the results in \cite{M}, the condition $L'(-1)\omega'=L^W(-1)\omega'$ can be removed from   \eqref{2.1} when we check whether a vector $\omega'$ is a semi-conformal vector of $(W,\omega^W)$.
\end{remark}
\begin{remark}\label{r2.3}  We can express the conditions (\ref{2.1}) by
operator product expansion(OPE) of vertex operators (see \cite{BF,K}) as follows
\begin{equation}\label{2.3}
L'(z)L'(w)\sim \frac{2L'(w)}{(z-w)^2}+\frac{\partial L'(w)}{(z-w)}+\frac{\frac{c'}{2}}{(z-w)^4};
\end{equation}
\begin{equation}\label{2.4}
L^W(z)L'(w)\sim \frac{2L'(w)}{(z-w)^2}+\frac{\partial L'(w)}{(z-w)}+\frac{\frac{c'}{2}}{(z-w)^4}.
\end{equation}
Sometimes, it is more convenient to use the operator product expansion (OPE) of vertex operators for verifying the conditions~\eqref{2.1}.
\end{remark}

\begin{proofof}{\bf Proof of Theorem 1.1.}  We will use Proposition~\ref{p2.2}. Each of the set of equations in \eqref{2.1} has two equalities (except the third one), the first and the second equalities. We will call them the first set and the second set of equations. We will show that each  of the two sets of equalities defines a Zariski closed subset of $V_2$. 
 
For each $ n \in \Z$, $V_n$ is finite dimensional. Hence $\Hom_{\C}(V_n, V_m)$ 
is a finite dimensional as well. Thus  for each $  l \in \N$ and each $n \in \Z$, 
the map $\rho_{n, l}: V_2 \rightarrow \Hom_{\C}(V_n, V_{n+1-l})$, defined by 
$v\mapsto \rho_{n,l}(v)=v_{l}: V_{n}\rightarrow V_{n+1-l}$, is a linear map and thus an algebraic map. 
Also $ L^V(n)=\omega^V_{n+1}: V_2\rightarrow V_{2-n}$ is linear. Thus the kernel 
$ \ker(L^V(n): V_2\rightarrow V_{2-n})$ is a Zariski closed subset of $ V_2$ for each $ n$. This shows that the set of vectors $\omega'\in V_2 $ satisfying the second set of equalities in \eqref{2.1}  is a Zariski closed subset of $V_2$. In fact it is an affine subspace of $V_2$. Here we have  used the fact  that the intersection of possibly infinitely many closed subsets is still closed.  

If  $\dim V_2=d$, we take  a basis  $\{e_1, \cdots, e_d\}$ of $V_2$. 
Now for each $ \omega'=\sum_{i=1}^d x_i e_i \in V_2$, $ \omega'_l(\omega') 
\in V_{3-l}$ can be expressed as a linear combination of a fixed basis of $ V_{3-l}$ 
with coefficients being quadratic polynomials of $\{x_1, \cdots, x_d\}$. 
Similarly, $\omega^V_l(\omega) \in V_{3-l}$ can be expressed as a linear 
combination of the same fixed basis of $ V_{3-l}$ with coefficients being 
linear polynomials of $\{x_1, \cdots, x_d\}$. Thus the first set of equalities in 
\eqref{2.1} gives a set of equations for a Zariski closed subset in $V_2$.

Thus the intersection of the closed subsets of $V_2$ defined by the two sets of  equations in \eqref{2.1} is a Zariski closed subset of $V_2$ which is exactly the set of all semi-conformal vectors in $V_2$ by Proposition~\ref{p2.2}. \qed
\end{proofof}
Next, we  define a partial ordering on  $\on{Sc}(V,\omega^V)$.
\begin{lem}\label{l2.5}
Let $\omega^1,\omega^2$ be two semi-conformal vectors of $(W,\omega^W)$. If there exists a semi-conformal
subalgebra $(U,\omega^2)$ such that $\omega^1\in (U,\omega^2)$, then
\begin{equation}\label{2.5}
C_W(C_W(<\omega^1>))\subset C_W(C_W(<\omega^2>)).
\end{equation}
\end{lem}
\begin{proof} Since $(<\omega^1>,\omega^1)\subset (U, \omega^2)$, then we have $C_{W}(<\omega^1>)\supset C_{W}(U)$
and $C_W(C_W(<\omega^1>))\subset C_W(C_W(<\omega^2>)).$
\end{proof}

\begin{prop}\label{p2.6} Let $(W,\omega^W)$ be a vertex operator algebra. If $\omega'\in \on{Sc}(W,\omega^W)$, then we have
$L^W(n)=L'(n)$ on $C_W(C_W(<\omega'>))$ for $n\geq -1$.
\end{prop}
\begin{proof} By using the proof of \cite[Theorem 3.11.12 ]{LL}, we know $\omega^W-\omega'$ is the conformal vector of the commutant $C_{W}(<\omega'>)$.
Since there is
\begin{eqnarray*} C_W(C_W(<\omega'>))&=&\{u\in W\;|\; v_nu=0, \forall v\in C_{W}(<\omega'>), n\geq 0\}\\ &=&\bigcap_{n\geq 0}\bigcap_{v\in C_{W}(<\omega'>)}\mbox{Ker}_{W}(v_n),
\end{eqnarray*} we have $$C_W(C_W(<\omega'>))
\subset \bigcap_{n\geq -1} \mbox{Ker}_{W}(L^W(n)-L'(n))$$ by taking $ v=\omega^W-\omega'$ and using $L^W(n)-L'(n)=v_{n+1}$. 
Thus,
$L^W(n)=L'(n)$ on $C_W(C_W(<\omega'>))$ for $n\geq -1$.
\end{proof}
\begin{defn}\label{d2.7}
For $\omega^1,\omega^2\in \on{Sc}(W,\omega^W)$, if $C_W(C_W(<\omega^1>))\subset C_W(C_W(<\omega^2>))$, we define $\omega^1\preceq \omega^2$
in $\on{Sc}(W,\omega^W)$. Then $\preceq$ gives a partial ordering on  $\on{Sc}(W,\omega^W)$.\end{defn}



\begin{prop}\label{p2.8}
For
$\omega^1,\omega^2\in \on{Sc}(W,\omega^W)$, $\omega^1\preceq \omega^2$ if and only if
the following conditions hold
\begin{eqnarray}\label{2.6}
\left\{
\begin{array}{lllll}
L^2(0)\omega^1=2\omega^1;\\
L^2(1)\omega^1=0;\\
L^2(2)\omega^1=L^1(2)\omega^1;\\
L^2(-1)\omega^1=L^1(-1)\omega^1;\\
L^2(n)\omega^1=0, ~\mbox{for}~ n\geq 3.
\end{array}
\right.
\end{eqnarray}
\end{prop}
\begin{proof}
Since $L(1)\omega^2=0$, we have
$$C_{W}(C_{W}(<\omega^2>))=\mbox{Ker}_{W}(L^W(-1)-L^2(-1)).$$
If the conditions (\ref{2.6}) are satisfied, then there is
$$L^W(-1)\omega^1=L^2(-1)\omega^1.$$
So $\omega^1\in \mbox{Ker}_{W}(L^W(-1)-L^2(-1))=C_{W}(C_{W}(<\omega^2>))$. By Proposition \ref{p2.2},
$\omega^1$ is a semi-conformal vector of $C_{W}(C_{W}(<\omega^2>))$,
that is, $\omega^1\preceq \omega^2$.

Conversely,  for $\omega^1,\omega^2\in \on{Sc}(W,\omega^W)$,  if $\omega^1\preceq \omega^2$, then we can verify
the relations (\ref{2.6}) by the proof of Proposition \ref{p2.2}.
\end{proof}

Next, we consider the commutant of
a vertex operator algebra as an operator on $\on{ScAlg}(W,\omega^W)$
as follows
 $$
 \begin{array}{lllll}
 C_W:\on{ScAlg}(W,\omega^W)\longrightarrow \on{ScAlg}(W,\omega^W)\\
 \hspace{2.4cm}(U,\omega')\longmapsto C_{W}(U).
 \end{array}
 $$
 Then the operator $C_W$ induces an involution $\omega_{-}^W$ of $\on{Sc}(W,\omega^W)$ as follows
 $$
 \begin{array}{lllll}
 \omega_{-}^W:\on{Sc}(W,\omega^W)\longrightarrow \on{Sc}(W,\omega^W)\\
 \hspace{2.8cm}\omega'\longmapsto \omega^W-\omega'.
 \end{array}
 $$
 In fact, associated with the surjection $\pi$
$$\begin{array}{llll}
\pi:\on{ScAlg}(W,\omega^W)\longrightarrow \on{Sc}(W,\omega^W)\\
\hspace{2.0cm}(U,\omega')\longmapsto \omega',
\end{array}
$$
we have the following commutative diagram
 $$\begin{CD}
 \on{ScAlg}(W,\omega^W)@>C_W>> \on{ScAlg}(W,\omega^W)\\
 @VV\pi V      @VV\pi V\\
 \on{Sc}(W,\omega^W)@>\omega_{-}^W>> \on{Sc}(W,\omega^W).
 \end{CD}
 $$
\begin{remark}
For the surjection $\pi:\on{ScAlg}(W,\omega^W)\longrightarrow \on{Sc}(W,\omega^W)$,
the fiber $\pi^{-1}(\omega')$ for each $\omega'\in \on{Sc}(W,\omega^W)$ is the set of semi-conformal subalgebras with the conformal vector $\omega'$ in $W$. In the fiber $\pi^{-1}(\omega')$, there exists a unique maximal one $C_{W}(C_{W}(<\omega'>))$( see \cite[Corollary 3.11.14]{LL}).
We can regard $\pi^{-1}(\omega')$ as the set of conformal subalgebras of $C_{W}(C_{W}(<\omega'>))$.
\end{remark}
Let $\mbox{Min}\on{Sc}(W,\omega^W)$ be the subset of $\on{Sc}(W,\omega^W)$ which consists of all non-zero minimal semi-conformal vectors of $(W,\omega^W)$ under the partial order $\preceq$.
For  $\omega^1,\omega^2\in \on{Sc}(W,\omega^W)$, we know that
$\omega^W-\omega^1$ and $\omega^W-\omega^2$ are  conformal vectors of commutants $C_{W}(<\omega^1>)$ and  $C_{W}(<\omega^2>)$, respectively. If $\omega^1\preceq \omega^2$, then $\omega^W-\omega^2\preceq \omega^W-\omega^1$. Hence we have the maximal subset of  $\on{Sc}(W,\omega^W)$ which consists of all non-trivial maximal semi-conformal vectors of $(W,\omega^W)$ under the partial order $\preceq$. Let $\mbox{Max}\on{Sc}(W,\omega^W)$ be the maximal subset of  $\on{Sc}(W,\omega^W)$. Then there are
\begin{equation}
\label{2.7}\mbox{Max}\on{Sc}(W,\omega^W)=\{\omega^W-\omega^1|\forall \omega^1\in \mbox{Min}\on{Sc}(W,\omega^W)\};\end{equation}
and
\begin{equation}
\label{2.8}
 \mbox{Min}\on{Sc}(W,\omega^W)=\{\omega^W-\omega^1|\forall \omega^1\in \mbox{Max}\on{Sc}(W,\omega^W)\}.
 \end{equation}

\section{ Semi-conformal vectors of the Heisenberg vertex operator algebra $(V_{\widehat{\h}}(1,0),\omega)$ }
\setcounter{equation}{0}
\subsection{} \label{sec3.1}
In the following, we  recall  some results of Heisenberg vertex operator algebras and refer to \cite{BF, LL} for
more details.

Let $\h$ be a $d$-dimensional vector space with a nondegenerate symmetric bilinear form $\langle\cdot,\cdot\rangle$.
$\hat{\h}=\C[t,t^{-1}]\otimes\h\oplus\C C$ is the affiniziation of
the abelian Lie algebra $\h$ defined by
\begin{align*}
[\beta_1\otimes t^{m},\,\beta_2\otimes
t^{n}]=m\langle\beta_1,\beta_2\rangle\delta_{m,-n}C\hbox{ and }[C,\hat{\h}]=0
\end{align*}
for  any $\beta_i\in\h (i=1,2),\,m,\,n\in\Z$. Then $\hat{\h}^{\geq
0}=\C[t]\otimes\h\oplus\C C$ is an Abelian subalgebra. For
$\forall \lambda\in\h$, we can define an one-dimensional $\hat{\h}^{\geq
0}$-module $\C e^\lambda$ by the actions $(h\otimes
t^{m})\cdot e^\lambda=(\lambda,h)\delta_{m,0}e^\lambda$ and
$C\cdot e^\lambda=e^\lambda$ for $h\in\h$ and $m\geq0$.
Set
\begin{align*}
V_{\widehat{\h}}(1,{\lambda})=U(\hat{\h})\otimes_{U(\hat{\h}^{\geq 0})}\C
e^\lambda\cong S(t^{-1}\C[t^{-1}]\otimes \h),
\end{align*}
which is an $\hat{\h}$-module induced from $\hat{\h}^{\geq 0}$-module $\C
e^\lambda$. When $ \lambda=0$,  let $\1=1\otimes e^0 \in V_{\widehat{\h}}(1,0)$. By the strong reconstruction theorem \cite[Thm. 4.4.1]{BF}, there is a unique vertex algebra structure 
$Y:V_{\widehat{\h}}(1,0)\to(\End (V_{\widehat{\h}}(1,{0})))[[z,z^{-1}]]$ on $V_{\widehat{\h}}(1,0)$. For each choice of  an orthonormal basis $\{h_1,\cdots, h_d\}$   of $\h$ and $\Lambda=(\Lambda_1,\cdots,\Lambda_d)\in \C^d$, define
$\omega_{\Lambda}=\frac{1}{2}\sum_{i=1}^{d}
h_i(-1)^2\cdot\mathbf{1}+\sum_{i=1}^{d}\Lambda_ih_i(-2)\cdot\1 \in V_{\widehat{\h}}(1,0)$.  Then $(V_{\widehat{\h}}(1,0),\,Y,\,\1,\,\w_{\Lambda})$  has a  vertex operator algebra
structure  and $(V_{\widehat{\h}}(1,{\lambda}),Y)$ becomes an irreducible module of 
$(V_{\widehat{\h}}(1,0),\omega_{\Lambda})$ with $\Lambda=0$ for any $\lambda\in\h$ (see \cite{FLM,LL}).

For different choices of $\Lambda\in \C^d$, the resulted vertex operator algebras are not isomorphic in general since the central charges can be different (see \cite[Examples 2.5.9]{BF})  although the underlying vertex algebra 
$(V_{\widehat{\h}}(1,0),\,Y,\,\1)$ is unique.  When $ \Lambda=0$, the vertex operator algebra$(V_{\widehat{\h}}(1,0),\,Y,\,\1,\,\w)$ is said to be the standard
Heisenberg operator algebra. It is easy to compute the automorphism group  
$\on{Aut}(V_{\widehat{\h}}(1,0),\,Y,\,\1,\,\w)=\on{O}(\h)$. This is not true 
when $ \Lambda\neq 0$.  Note that the computations in \cite[Section 2.5.9]{BF} shows 
that $L^{\omega_{\Lambda}}(0)=(\omega_{\Lambda})_1$ is always the degree operator. Hence the gradation on $V_{\widehat{\h}}(1,0)$ are all the same for all different $\Lambda$ and $ V_{\widehat{\h}}(1,0)_{n}$ is independent of the choice of $\Lambda$. 

\subsection{}
For the Heisenberg vertex operator algebra  $(V_{\widehat{\h}}(1,0),\omega_\Lambda)$,  $V_{\widehat{\h}}(1,0)_2$  has a basis
\begin{equation}
\label{f3.1}\{h_{i}(-1)h_{j}(-1)\cdot\mathbf{1}; h_{k}(-2)\cdot\mathbf{1}|1 \leq i\leq j\leq d, k=1,\cdots, d\}.\end{equation}
Let \begin{equation}\label{f3.2}\omega'=\sum_{1 \leq i\leq j\leq d}a_{ij}h_{i}(-1)h_{j}(-1)\cdot\mathbf{1}+\sum_{i=1}^{d}b_ih_{i}(-2)\cdot\mathbf{1}\in V_{\widehat{\h}}(1,0)_2.\end{equation} Then there exists a unique symmetric matrix
\begin{align}\label{f3.3}
A_{\omega'}=
\left(\begin{array}{cccccc}
&2a_{11} &a_{12} &\cdots &a_{1d}\\
&a_{12}& 2a_{22}&\cdots& a_{2d}\\
&\cdots&\cdots&\cdots&\cdots\\
&a_{1d}& \cdots &a_{d-1d} &2a_{dd}
\end{array}\right)
\end{align}
and a vector
$B_{\omega'}=(b_1,\cdots,b_d)^{tr}$ with entries in $\C$ such that
\begin{eqnarray}
\begin{array}{llll}
\omega'&=\frac{1}{2}(h_1(-1),\cdots, h_d(-1))A_{\omega'}
\left(\begin{array}{c}
h_1(-1)\\
~~~~~~\vdots\\
 h_d(-1)
\end{array}\right)\cdot\mathbf{1}+(h_1(-2),\cdots, h_d(-2)
)B_{\omega'}\cdot\1,
\end{array}
\end{eqnarray}
where the symbol $C^{tr}$ is the transpose of  the  matrix $C$.

\begin{prop}\label{p3.1}
$\omega'\in \on{Sc}(V_{\widehat{\h}}(1,0),\omega_{\Lambda})$ if and only if $A_{\omega'},B_{\omega'}$ satisfy
$$A_{\omega'}^{tr}=A_{\omega'},\; A_{\omega'}^2=A_{\omega'},\; A_{\omega'}\Lambda^{tr}=B_{\omega'}^{tr}.
$$ 
\end{prop}
\begin{proof} By Proposition \ref{p2.2} and Remark \ref{r2.3}, we can compute that
$\omega'\in \on{Sc}(V_{\widehat{\h}}(1,0),\omega_{\Lambda})$ if and only if $A_{\omega'},B_{\omega'}$ satisfy
\begin{eqnarray}
\label{3.3}
\left\{
\begin{array}{llllllll}
4a_{11}^2+a_{12}^2+\cdots+a_{1d}^2=2a_{11};\\
\ \cdots\ \ \ \ \ \cdots\ \ \ \ \ \ \cdots\\
a_{1d}^2+\cdots+a_{d-1d}^2+4a_{dd}^2=2a_{dd};\\
2a_{11}b_1+a_{12}b_2+\cdots+a_{1d}b_d=b_1;\\
a_{12}b_1+a_{22}b_2+\cdots+a_{2d}b_d=b_2;\\
\cdots\ \ \ \ \ \cdots\ \ \ \ \ \cdots\ \ \ \ \ \cdots\\
a_{1d}b_1+\cdots+a_{d-1d}b_{r-1}+2a_{dd}b_d=b_d;\\
2\Lambda_1 a_{11}+\Lambda_2a_{12}+\cdots+\Lambda_d a_{1d}=b_1;\\
\Lambda_1a_{12}+2\Lambda_2a_{22}+\cdots+\Lambda_da_{2d}=b_2;\\
\cdots\ \ \ \ \ \cdots\ \ \ \ \ \cdots\ \ \ \ \ \cdots\\
\Lambda_1a_{1d}+\cdots+\Lambda_{d-1}a_{d-1d}+2\Lambda_da_{dd}=b_d;\\
\sum\limits_{i=1}^{d}a_{ii}-6\sum\limits_{i=1}^{d}\Lambda_ib_i=\sum\limits_{1\leq i<j\leq d}a_{ij}^2+2\sum\limits_{i=1}^{d}a_{ii}^2-6\sum\limits_{i=1}^{d}b_i^2;\\
a_{i1}a_{1j}+\cdots+a_{ii-1}a_{i-1j}+2a_{ii}a_{ij}+a_{ii+1}a_{i+1j}+\cdots+a_{ij-1}a_{j-1j}\\
+2a_{ij}a_{jj}+a_{ij+1}a_{j+1j}+\cdots+a_{id}a_{dj}=a_{ij},\mbox{for}~ 1\leq i< j\leq d.\end{array}
\right.
\end{eqnarray}
 In fact, the relations \eqref{3.3} can be rewritten as follows:
\begin{equation}\label{3.7}
 A_{\omega'}^2=A_{\omega'},\; A_{\omega'}B_{\omega'}=B_{\omega'}, \; A_{\omega'}\Lambda^{tr}=B_{\omega'},\; \Lambda B_{\omega'}=B_{\omega'}^{tr}B_{\omega'}.
\end{equation}
Assuming $A_{\omega'}=A_{\omega'}^{tr},\; A_{\omega'}^2=A_{\omega'},\; A_{\omega'}\Lambda^{tr}=B_{\omega'}$, then there are
$$A_{\omega'}B_{\omega'}=A_{\omega'}(A_{\omega'}\Lambda^{tr})=A_{\omega'}^2\Lambda^{tr}=A_{\omega'}\Lambda^{tr}=B_{\omega'};$$
$$\Lambda B_{\omega'}=\Lambda A_{\omega'}\Lambda^{tr}=\Lambda A_{\omega'}^2\Lambda^{tr}=(A_{\omega'}\Lambda^{tr})^{tr}A_{\omega'}\Lambda^{tr}
=B^{tr}_{\omega'}B_{\omega'}.$$
Therefore, conditions in \eqref{3.7}  under the assumption that $A$ is symmetric is equivalent to the following three relations:
$$A_{\omega'}^{tr}=A_{\omega'},\; A_{\omega'}^2=A_{\omega'},\; A_{\omega'}\Lambda^{tr}=B_{\omega'}.
$$
\end{proof}

Let $ G_{\Lambda}=\on{Aut}(V_{\widehat{\h}}(1,0),\omega_\Lambda)$ be the automorphism group of the vertex operator algebra $(V_{\widehat{\h}}(1,0),\omega_\Lambda)$. Then $G_{\Lambda}$ is a subgroup of $\on{O}(\h)$. $G_{\Lambda}$ acts on $ V_{\widehat{\h}}(1,0)_n$ for all $n$. In particular, $G_{\Lambda}$ acts on the algebraic variety $ \on{Sc}(V_{\widehat{\h}}(1,0),\omega_\Lambda)$. One of the questions is to determine the $G_\Lambda$-orbits in $\on{Sc}(V_{\widehat{\h}}(1,0),\omega_\Lambda)$. In this paper, we will concentrate on the case $\Lambda=0$ and the general cases will be pursued in a later work. 

 When $\Lambda=0$, we denote $\omega_{\Lambda}$ by $\omega$. Thus the conformal automorphism group $G=\on{Aut}(V_{\widehat{\h}}(1,0), \omega)=\on{O}(\h)$, which is  the orthogonal group with respect to the fixed non-degenerate symmetric bilinear form on $\h$.    With respect to a fixed orthonormal basis $\{ h_1, \cdots, h_d\}$ of $\h$, the symmetric matrix $A_{\omega'}$ defines  a self adjoint  (with respect to the symmetric bilinear form on $\h$)  linear transformation $\mathcal{A}_{\omega'}$ of $\h$. By Proposition 3.1, we get that $\omega'\in \on{Sc}(V_{\widehat{\h}}(1,0),\omega)$ is one to one correspondence to a self adjoint  idempotent linear transformation $\mathcal{A}_{\omega'}$ of $\h$. Thus, the set $\on{Sc}(V_{\widehat{\h}}(1,0),\omega)$ can be described as the set of  self adjoint idempotent linear transformations of $\h$. Let $$\on{SymId}(\h)=\{\mathcal{A}\in \on{End}(\h)\;|\;
\mathcal{A}^2=\mathcal{A}\; \text{~and~}\;  (\mathcal{A}(u),v)=(u,\mathcal{A}(v)),  \forall u,v\in \h\}.$$
 Then we have
\begin{prop}\label{p3.2}
The map $\omega'\mapsto  \mathcal{A}_{\omega'} $ is a bijection from $\on{Sc}(V_{\widehat{\h}}(1,0),\omega)$  to $\on{SymId}(\h)$.
\end{prop}

If the restriction $(\cdot,\cdot)|_{\h'}$ on $\h$ to $\h'$ is still nondegenerate,  we say $\h'$ is {\em a regular subspace} of $\h$. Let 
$\on{Reg}(\h)=\{\h'\subset \h\;|\;\h'\; \text{ is a regular subspace of  }\; \h\}$. Then we have
\begin{cor} \label{c3.3}
 The map $\omega'\mapsto\on{Im}\mathcal{A}_{\omega'} $ is a bijection from $\on{Sc}(V_{\widehat{\h}}(1,0),\omega)$ to  $\on{Reg}(\h)$.
\end{cor}
\begin{proof}
For $\omega'\in \on{Sc}(V_{\widehat{\h}}(1,0),\omega)$, there is a unique $\mathcal{A}_{\omega'} \in \on{SymId}(\h)$.
Hence we have $\h=\on{Im}\mathcal{A}_{\omega'}\oplus \on{Ker}\mathcal{A}_{\omega'}$,
where $\on{Ker}\mathcal{A}_{\omega'}=\on{Im}\mathcal{A}_{\omega'}^{\bot}$.
Since the restriction $(\cdot,\cdot)$ on $\h$ to $\on{Im}\mathcal{A}_{\omega'}$ is still nondegenerate, then by Proposition 3.2, we know that $\omega'\mapsto
\mathcal{A}_{\omega'}$ gives a subspace $\on{Im}\mathcal{A}_{\omega'}\in \on{Reg}(\h)$. Conversely, let $\h'$ be a regular subspace of $\h$. There exists a projection $\mathcal{A}\in \on{End}(\h)$ such that $\on{Im}(\mathcal{A})=\h'$ and $(\mathcal{A}(u),v)=(u,\mathcal{A}(v))~for~u,v\in \h$.
Thus, $\on{Reg}(\h)\subset \on{Sc}(V_{\widehat{\h}}(1,0),\omega)$. By Proposition 3.2, we know that there exists a unique
$\omega'\in \on{Sc}(V_{\widehat{\h}}(1,0),\omega)$ such that $\mathcal{A}_{\omega'}=\mathcal{A}$. Thus, $\h'\mapsto \omega'$ gives the
inverse of the map $\omega'\mapsto\on{Im}\mathcal{A}_{\omega'}$.
\end{proof}

Next, we shall find a partial ordering on $\on{SymId}(\h)$.
\begin{defn}\label{def3.2}
Let $\mathcal{A}_1,\mathcal{A}_2\in\on{SymId}(\h)$. If they satisfy $ \mathcal{A}_2\mathcal{A}_1=\mathcal{A}_1\mathcal{A}_2=\mathcal{A}_2$, we define  $\mathcal{A}_2\leq \mathcal{A}_1$. Then $\leq$ is a partial
ordering on $\on{SymId}(\h)$.
\end{defn}
\begin{lem}\label{p3.5}
For $\mathcal{A}_1,\mathcal{A}_2\in \on{SymId}(\h)$, $\mathcal{A}_1\leq \mathcal{A}_2$ if and only if $\on{Im}\mathcal{A}_1\subset \on{Im}\mathcal{A}_2$.
\end{lem}

\begin{proof}
For $\mathcal{A}_1,\mathcal{A}_2\in \on{SymId}(\h)$, $\mathcal{A}_1\leq \mathcal{A}_2$, i.e,
$\mathcal{A}_1\mathcal{A}_2=\mathcal{A}_2\mathcal{A}_1=\mathcal{A}_1.$
Then we have
$\on{Im}\mathcal{A}_1=\on{Im}\mathcal{A}_2\mathcal{A}_1.$
And since $\on{Im}\mathcal{A}_2\mathcal{A}_1\subset \on{Im}\mathcal{A}_2$, then $\on{Im}\mathcal{A}_1\subset \on{Im}\mathcal{A}_2$.

Conversely, for  $\mathcal{A}_1,\mathcal{A}_2\in \on{SymId}(\h)$,
 we have $\h=\on{Ker}\mathcal{A}_1\oplus \on{Im}\mathcal{A}_1=\on{Ker}\mathcal{A}_2\oplus \on{Im}\mathcal{A}_2.$
For $\forall \alpha\in \h$, there is
$$\alpha=\alpha_1+\alpha_2=\beta_1+\beta_2,$$
where $\alpha_1\in \on{Ker}\mathcal{A}_1,\alpha_2\in\on{Im}\mathcal{A}_1;\beta_1\in \on{Ker}\mathcal{A}_2,\beta_2\in \on{Im}\mathcal{A}_2.$ Since $\mathcal{A}_2(\beta_2)=\beta_2$, we have
$$
\mathcal{A}_1\mathcal{A}_2(\alpha)=\mathcal{A}_1(\mathcal{A}_2(\beta_1+\beta_2))
=\mathcal{A}_1(\mathcal{A}_2(\beta_2))=\mathcal{A}_1(\beta_2).
$$
If $\on{Im}\mathcal{A}_1\subset \on{Im}\mathcal{A}_2$, then $\on{Ker}\mathcal{A}_2\subset
\on{Ker}\mathcal{A}_1$. So we have $\mathcal{A}_1(\alpha)=\mathcal{A}_1(\beta_1+\beta_2)=\mathcal{A}_1(\beta_2).$ Thus, we get $\mathcal{A}_1\mathcal{A}_2=\mathcal{A}_1$.

Similarly, if $\on{Im}\mathcal{A}_1\subset \on{Im}\mathcal{A}_2$, we have
$$
\mathcal{A}_2\mathcal{A}_1(\alpha)=\mathcal{A}_2(\mathcal{A}_1(\alpha_1+\alpha_2))
=\mathcal{A}_2(\mathcal{A}_1(\alpha_2))=\mathcal{A}_1(\alpha_2)=\mathcal{A}_1(\alpha).
$$
And then $\mathcal{A}_2\mathcal{A}_1=\mathcal{A}_1$. Thus, we get $\mathcal{A}_1\leq \mathcal{A}_2$.\end{proof}
\begin{prop} Under the bijection in Proposition~\ref{p3.2},
$\omega^1\preceq\omega^2$ in $\on{Sc}(V_{\widehat{\h}}(1,0),\omega)$ if and only if
 $\on{Im}\mathcal{A}_{\omega^1}\subset \on{Im}\mathcal{A}_{\omega^2}.$
\end{prop}
 \begin{proof}Let $\mathcal{A}_{\omega^1},\mathcal{A}_{\omega^2}\in \on{SymId}(\h)$ be the  linear transformations of $\h$ corresponding to $\omega^1,\omega^2$, respectively.
 For convenience, we write $\mathcal{A}_{\omega^1}, \mathcal{A}_{\omega^2}$ as
 $\mathcal{A}_1, \mathcal{A}_2$, respectively.
 If  $\omega^1\preceq \omega^2 $, that is, $\omega^1$ is a semi-conformal vector of the subalgebra $C_{V_{\widehat{\h}}(1,0)}(C_{V_{\widehat{\h}}(1,0)}(<\omega^2>))$, then by Proposition 2.8 and an argument similar to the proof of Proposition 3.1,  we have the relation $\mathcal{A}_1\mathcal{A}_2+\mathcal{A}_2\mathcal{A}_1=2\mathcal{A}_1.$
Hence we have
$$(\mathcal{A}_1\mathcal{A}_2+\mathcal{A}_2\mathcal{A}_1)\mathcal{A}_1
=\mathcal{A}_1\mathcal{A}_2\mathcal{A}_1+\mathcal{A}_2\mathcal{A}_1\mathcal{A}_1
=\mathcal{A}_1\mathcal{A}_2\mathcal{A}_1+\mathcal{A}_2\mathcal{A}_1
=2\mathcal{A}_1^2=2\mathcal{A}_1;$$
$$\mathcal{A}_1(\mathcal{A}_1\mathcal{A}_2+\mathcal{A}_2\mathcal{A}_1)
=\mathcal{A}_1\mathcal{A}_1\mathcal{A}_2+\mathcal{A}_1\mathcal{A}_2\mathcal{A}_1
=\mathcal{A}_1\mathcal{A}_2+\mathcal{A}_1\mathcal{A}_2\mathcal{A}_1=2\mathcal{A}_1^2=2\mathcal{A}_1.$$
Thus,
$\mathcal{A}_1\mathcal{A}_2=\mathcal{A}_2\mathcal{A}_1$. Then we have $\mathcal{A}_1\mathcal{A}_2=\mathcal{A}_2\mathcal{A}_1=\mathcal{A}_1$, i.e, $\mathcal{A}_1\leq \mathcal{A}_2$.

Conversely, if $\mathcal{A}_1,\mathcal{A}_2\in \on{SymId}(\h)$ and $\mathcal{A}_1\leq \mathcal{A}_2$, then we have $\mathcal{A}_1\mathcal{A}_2=\mathcal{A}_2\mathcal{A}_1=\mathcal{A}_1.$
 Let $\omega^1,\omega^2\in \on{Sc}(V_{\widehat{\h}}(1,0),\omega)$ be semi-conformal vectors  corresponding to  $\mathcal{A}_1,\mathcal{A}_2$, respectively. Again using an argument  similar to the proof of Proposition 3.1, we can get the relations (\ref{2.6}), i.e, $\omega^1\preceq \omega^2$.
 \end{proof}
 By above proposition, we know that  the partial ordering $\subset$ (the inclusion relation) of $\on{Reg}(\h)$ coincides with  the partial ordering $\preceq$ on $\on{Sc}(V_{\widehat{\h}}(1,0),\omega)$.

%
%
\begin{proofof}{\bf Proof of Theorem 1.2.}
According to  Corollary 3.3 and Proposition 3.6, we know the bijection $\rho$ preserves orders between $\on{Sc}(V_{\widehat{\h}}(1,0),\omega)$ and $\on{Reg(\h)}$.

As we know, $\on{Aut}(V_{\widehat{\h}}(1,0),\omega)=\on{O}(\h)$, this group acts on  $\on{Sc}(V_{\widehat{\h}}(1,0),\omega)$ as follows:
 for a $\omega_A\in \on{Sc}(V_{\widehat{\h}}(1,0),\omega)$ and $\sigma\in  \on{Aut}(V_{\widehat{\h}}(1,0),\omega)$,  we have $ \sigma(V_2)=V_2$ and $ \sigma $ preserves the bilinear form on $\h$. Let $o\in \on{O}_d(\C)$ be the matrix  of $ \sigma$ with respect to the fixed orthonormal basis, then by the definition, we have $\sigma(\omega_A)=\omega_{oAo^{tr}}\in \on{Sc}(V_{\widehat{\h}}(1,0),\omega)$. Thus $ \on{Im}(oAo^{\tr})=\sigma(\on{Im}(A))$.

For a $\h'\in \on{Reg}(\h)$ and $\sigma\in  \on{Aut}(V_{\widehat{\h}}(1,0),\omega)$, $\sigma(\h')\in \on{Reg}(\h)$.
 This gives an action of $\on{Aut}(V_{\widehat{\h}}(1,0),\omega)$ on $\on{Reg}(\h)$.

According to the actions of  $\on{Aut}(V_{\widehat{\h}}(1,0),\omega)$ on $\on{Sc}(V_{\widehat{\h}}(1,0),\omega)$ and $\on{Reg}(\h)$, we know that the above bijection also preserves the actions of  $\on{Aut}(V_{\widehat{\h}}(1,0),\omega)$ on $\on{Sc}(V_{\widehat{\h}}(1,0),\omega)$ and $\on{Reg}(\h)$. So the first assertion 1) holds.

According to the first assertion  1), we can consider the action of $\on{Aut}(V_{\widehat{\h}}(1,0),\omega)$ on $\on{Sc}(V_{\widehat{\h}}(1,0),\omega)$  by the action of $\on{Aut}(V_{\widehat{\h}}(1,0),\omega)$ on $\on{Reg}(\h)$. Then we have the second assertion.

The third  assertion follows from the  second assertion. \qed
\end{proofof}
\section{Conformally closed semi-confromal subalgebras  of the Heisenberg vertex operator algebra $(V_{\widehat{\h}}(1,0),\omega)$.}
In this section,  we shall describe the conformally closed semi-conformal subalgebras of the Heisenberg vertex operator algebra $(V_{\widehat{\h}}(1,0), \omega=\omega_{\Lambda})$ with  $ \Lambda=0$.

For $\omega' \in \on{Sc}(V_{\widehat{\h}}(1,0),\omega)$ as defined in \eqref{f3.2} with the corresponding symmetric idempotent matrix $A_{\omega'}$ given in \eqref{f3.3}. In this case, $B_{\omega'}=0$ in \eqref{f3.2}.  By Theorem~\ref{thm1.2}, we know that
there is a  regular subspace $\on{Im}\mathcal{A}_{\omega'}$ corresponding to $\omega'$.
\begin{prop}
\label{p3.7}
$$\on{Im}\mathcal{A}_{\omega'}=\on{Ker}_{V_{\widehat{\h}}(1,0)}(L(-1)-L'(-1))\cap V_{\widehat{\h}}(1,0)_1, \; \on{Ker}\mathcal{A}_{\omega'}= \on{Ker}_{V_{\widehat{\h}}(1,0)} L'(-1)\cap  V_{\widehat{\h}}(1,0)_1. $$
\end{prop}
\begin{proof}
By Proposition 3.11.11 and Theorem 3.11.12 in \cite{LL}, we have
$$ \on{Ker}_{V_{\widehat{\h}}(1,0)} L'(-1)=C_{V_{\widehat{\h}}(1,0)}(<\omega'>);$$
 $$\on{Ker}_{V_{\widehat{\h}}(1,0)}(L(-1)-L'(-1))=C_{V_{\widehat{\h}}(1,0)}(C_{V_{\widehat{\h}}(1,0)}(<\omega'>)).$$
Then $\on{Ker}_{V_{\widehat{\h}}(1,0)}L'(-1)_{1}\cap V_{\widehat{\h}}(1,0)_1\subset  V_{\widehat{\h}}(1,0)_1\cong \h$.

With  $\omega' \in \on{Sc}(V_{\widehat{\h}}(1,0),\omega)$  written as in  \eqref{f3.2} and $B_{\omega'}=0$,
we have
$$L'(-1)=\sum\limits_{i=1}^{d}a_{ii}\sum\limits_{k\in \Z}:h_i(k)h_i(-1-k):+\sum\limits_{1\leq i<j\leq d}a_{ij}
\sum\limits_{k\in \Z}:h_i(k)h_j(-1-k):,$$
where $:\cdots:$ is the normal order product.
Hence
\begin{equation}\label{3.14} L(-1)-L'(-1)=\sum\limits_{i=1}^{d}(\frac{1}{2}-a_{ii})\sum\limits_{k\in \Z}:h_i(k)h_i(-1-k):-\sum\limits_{1\leq i<j\leq d}a_{ij}
\sum\limits_{k\in \Z}:h_i(k)h_j(-1-k):.\end{equation}

For $\forall h\in V_{\widehat{\h}}(1,0)_{1}$, set $h=\sum\limits_{m=1}^{d}a_mh_{m}(-1)\cdot\mathbf{1}$. Using the formula (\ref{3.14}), we have that
$h\in \on{Ker}_{V_{\widehat{\h}}(1,0)}(L(-1)-L'(-1))\cap V_{\widehat{\h}}(1,0)_1$ if and only if the vector $(a_1,a_2,\cdots,a_d)$ is a solution of the linear equation system \begin{eqnarray}\label{3.15}\left\{
 \begin{array}{llll}
 (1-2a_{11})x_1-a_{12}x_2-\cdots-a_{1d}x_d=0,\\
 -a_{12}x_1+(1-2a_{22})x_2-\cdots-a_{2d}x_d=0,\\
 \ \cdots\ \ \ \ \cdots\ \ \ \ \cdots\ \ \ \ \cdots\ \ \ \ \cdots\ \ \ \ \cdots\\
 -a_{1d}x_1-a_{2d}x_2-\cdots+(1-2a_{dd})x_d=0.
 \end{array}
 \right.
 \end{eqnarray}
Let $X=(x_1,x_2,\cdots,x_d)^{tr}$. Then the linear equation system
(\ref{3.15}) becomes the matrix equation
$(I-A_{\omega'})X=0,$ where $I$ is the $d\times d$ identity matrix. So we have $\on{Im}\mathcal{A}_{\omega'}=\on{Ker}_{V_{\widehat{\h}}(1,0)}(L(-1)-L'(-1))\cap V_{\widehat{\h}}(1,0)_1$. Note that  the corresponding matrix for $ \omega-\omega'$ is $A_{\omega-\omega'}=I-A_{\omega'}$.   Hence  with $\omega'$ replaced by $\omega-\omega'$, we have $ \on{Ker}\mathcal{A}_{\omega'}= \on{Ker}_{V_{\widehat{\h}}(1,0)} L'(-1)\cap  V_{\widehat{\h}}(1,0)_1.$
\end{proof}


%
%

\begin{proofof}{\bf Proof of Theorem 1.3.} By Proposition 4.1, we have
$$C_{V_{\widehat{\h}}(1,0)}(C_{V_{\widehat{\h}}(1,0)}(<\omega'>))_1\cong \on{Im}\mathcal{A}_{\omega'}\quad \text{and} \quad C_{V_{\widehat{\h}}(1,0)}(<\omega'>)_1\cong\on{Ker}\mathcal{A}_{\omega'}.$$
Since $\on{Im}\mathcal{A}_{\omega'}$  and $\on{Ker}\mathcal{A}_{\omega'}$ are both regular subspaces of $\h$. Then
$$<\on{Im}\mathcal{A}_{\omega'}>\cong V_{\widehat{\on{Im}\mathcal{A}_{\omega'}}}(1,0)\quad \text{and} \quad
<\on{Ker}\mathcal{A}_{\omega'}>\cong V_{\widehat{\on{Ker}\mathcal{A}_{\omega'}}}(1,0)$$
in $V_{\widehat{\h}}(1,0)$.
Since $A_{\omega'}^2=A_{\omega'}$, we have
\begin{eqnarray}
\begin{array}{llll}
\omega'&=\frac{1}{2}(h_1(-1),\cdots, h_d(-1))A_{\omega'}
\left(\begin{array}{lllll}
h_1(-1)\\
~~~~~~\vdots\\
 h_d(-1)
\end{array}\right)\cdot\mathbf{1}\\
&=\frac{1}{2}
\left(A_{\omega'}
\left(\begin{array}{lllll}
h_1(-1)\\
~~~~~~~~~~~~~~~~~~~~\vdots\\
 h_d(-1)
\end{array}\right)\right)^t
\left(A_{\omega'}
\left(\begin{array}{lllll}
h_1(-1)\\
~~~~~~~~~~~~~~~~~~~~\vdots\\
 h_d(-1)
\end{array}\right)\right)\cdot\mathbf{1}.
\end{array}
\end{eqnarray}
 Then  $\omega'$ is the conformal vector of $<\on{Im}\mathcal{A}_{\omega'}>$.

Similarly, we have
$$\begin{array}{llllll}
\omega-\omega'&=
\frac{1}{2}(h_1(-1),\cdots, h_d(-1))(I-A_{\omega'})
\left(\begin{array}{lllll}
h_1(-1)\\
~~~~~~\vdots\\
 h_d(-1)
\end{array}\right)\cdot\mathbf{1}\\
&=\frac{1}{2}(h_1(-1),\cdots, h_d(-1))(I-A_{\omega'})^2
\left(\begin{array}{lllll}
h_1(-1)\\
~~~~~~\vdots\\
 h_d(-1)
\end{array}\right)\cdot\mathbf{1}\\
&=\frac{1}{2}
\left((I-A_{\omega'})
\left(\begin{array}{lllll}
h_1(-1)\\
~~~~~~\vdots\\
 h_d(-1)
\end{array}\right)\right)^t
\left((I-A_{\omega'})
\left(\begin{array}{lllll}
h_1(-1)\\
~~~~~~\vdots\\
 h_d(-1)
\end{array}\right)\right)\cdot\mathbf{1}.
\end{array}$$
Hence $\omega-\omega'$  is the conformal vector of $<\on{Ker}\mathcal{A}_{\omega'}>$. 

Since $\h=\on{Im}\mathcal{A}_{\omega'}\oplus \on{Ker}\mathcal{A}_{\omega'}$ and $\on{Ker}\mathcal{A}_{\omega'}=(\on{Im}\mathcal{A}_{\omega'})^{\perp}$,
then there is
$$V_{\widehat{\h}}(1,0)\cong <\on{Im}\mathcal{A}_{\omega'}>\otimes <\on{Ker}\mathcal{A}_{\omega'}>.$$
Thus we have \begin{equation}\label{3.17} C_{V_{\widehat{\h}}(1,0)}(<\omega'>)= V_{\widehat{\on{Ker}\mathcal{A}_{\omega'}}}(1,0)\quad \text{and}\end{equation}
\begin{equation} \label{3.18} C_{V_{\widehat{\h}}(1,0)}(C_{V_{\widehat{\h}}(1,0)}(<\omega'>))=C_{V_{\widehat{\h}}(1,0)}(<\omega-\omega'>)
=V_{\widehat{\on{Im}\mathcal{A}_{\omega'}}}(1,0).\end{equation}
Moreover, we have
$$C_{V_{\widehat{\h}}(1,0)}(V_{\widehat{\on{Ker}\mathcal{A}_{\omega'}}}(1,0))=V_{\widehat{\on{Im}\mathcal{A}_{\omega'}}}(1,0)\;\text{and}\;  C_{V_{\widehat{\h}}(1,0)}(V_{\widehat{\on{Im}\mathcal{A}_{\omega'}}}(1,0))=V_{\widehat{\on{Ker}\mathcal{A}_{\omega'}}}(1,0).$$
Finally, we can also get that
$$V_{\widehat{\h}}(1,0)\cong C_{V_{\widehat{\h}}(1,0)}(<\omega'>)\otimes C_{V_{\widehat{\h}}(1,0)}(C_{V_{\widehat{\h}}(1,0)}(<\omega'>)).$$\qed
\end{proofof} 
\begin{remark}
From above results, we know that all  of conformally closed semi-conformal subalgebras in $V_{\widehat{\h}}(1,0)$  are  Heisenberg vertex operator algebras generated by regular subspaces of the weight-one subspace $V_{\widehat{\h}}(1,0)_1$. 
\end{remark}
\section{Characterizations  of Heisenberg vertex operator algebras}
 In this section, we will use the structures of $\on{Sc}(V,\omega)$ and $\on{ScAlg}(V,\omega)$ to give two characterizations of Heisenberg vertex operator algebras. Let us fix the notation $ Y(u, z)=\sum_{n} u_n z^{-n-1}$ for vertex operators. However, in case  that the vertex operator algebra is defined by a Lie algebra $\h$, we will use the notation $ h(n)=h\otimes t^{n}$ in the algebra $\h[t,t^{-1}]$ with $h \in \h$. 

Let $V$ be a simple $\N$-graded vertex operator algebra  with $V_0=\C 1$. Such  $V$ is also called $a~simple~CFT~type$ vertex operator algebra (\cite{DLMM,DM}).  If $V$ satisfies the further condition
that $L(1)V_1=0$, it is called  {\em strong~CFT~type}. 
Li has shown (\cite{Li}) that such a vertex operator algebra $V$ 
has a unique nondegenerate invariant  bilinear form 
$\langle\cdot,\cdot\rangle$ up to a multiplication of a nonzero scalar. 
In particular, the restriction of $\langle\cdot,\cdot\rangle$ to $V_1$ endows $V_1$ 
with a nondegenerate symmetric invariant  bilinear form  defined by
$u_1(v)=\langle u,v\rangle \mbf{1}$ for $u,v\in V_1$.  For $v\in V_n$, the component operator $v_{n-1}$ is called the zero mode of $v$.  It is well-known that $V_1$ forms a Lie algebra with the bracket operation $[u,v]=u_0(v)$ for $u,v\in V_1$.

  In this paper, we will consider  vertex operator algebras $(V,\omega)$  that satisfy the following conditions:
(1) $V$ is a {\em simple CFT type} vertex operator algebra;
(2) The symmetric bilinear form $\langle u,v\rangle=u_1v$ for $u,v\in V_1$ is  nondegenerate. For convenience, we call  such a vertex operator algebra $(V,\omega)$   {\em nondegenerate simple CFT type}. We note that for any vertex operator algebra $(V, \omega)$, and any $\omega' \in \on{Sc}(V,\omega)$, one has
$C_{V}(C_{V}(<\omega'>))\otimes C_{V}(<\omega'>)\subseteq V$ as a conformal vertex operator subalgebra.

\begin{lem}\label{l6.1}\label{lem5.1}
 Let $(V,\omega)$ be a nondegenerate simple CFT type vertex operator algebra generated by $V_1$. Let 
 $(V', \omega')$ and $(V'', \omega'')$ be two vertex operator subalgebras with possible different conformal vectors. Assume that $(V, \omega)=(V', \omega')\otimes (V'', \omega'')$ is a tensor product of vertex operator algebras (see \cite[Section 3.12]{LL}).  Then
\begin{itemize}
\item [1)] $(V',\omega')$ and $(V'' ,\omega'')$ are semi-conformal subalgebras and  both are also non-degenerate simple CFT type;
\item [2)] $ V_1=V_1'\otimes \mathbf{1}''\oplus \mathbf{1}'\otimes V_1'',$ is an orthogonal decomposition with
respect to the non-degenerate symmetric bilinear form $\langle\cdot,\cdot\rangle$ on $V_1$;

\item [3)] $[V_1'\otimes \mathbf{1}'',\mathbf{1}'\otimes V_1'']=0$ with the Lie bracket $[\cdot,\cdot]$ on $V_1$;

\item [4)]  $\on{Sc}(V',\omega')\otimes \mbf{1}''$,  $ \mbf{1}'\otimes\on{Sc}(V'',\omega'')$, and $\on{Sc}(V',\omega')\otimes\mbf{1}''+\mbf{1}'\otimes\on{Sc}(V'',\omega'')$ are subsets of $\on{Sc}(V,\omega);$

\item [5)] For each $\widetilde{\omega}'\in \on{Sc}(V',\omega')$, we have
$$C_{V}(<\widetilde{\omega}'>\otimes \mbf{1}'')=C_{V'}(<\widetilde{\omega}'>)\otimes V''$$
and
$$C_{V}(C_{V}(<\widetilde{\omega}'>\otimes \mbf{1}''))=C_{V'}(C_{V'}(<\widetilde{\omega}'>))\otimes \mbf{1}''.$$
\end{itemize}
\end{lem}
\begin{proof}  The first assertion is straight forward to verify using properties of tensor product vertex operator algebras. 

For $u'\in V_1', v''\in V_1''$, we have
$$ (u'\otimes \mathbf{1}'')_1(\mathbf{1}'\otimes v'')=u'_1\mathbf{1}'\otimes v''=\langle u'\otimes \mathbf{1}'',\mathbf{1}'\otimes v''\rangle=0.$$
So  $V_1'\otimes \mathbf{1}''$ is orthogonal to  $\mathbf{1}'\otimes V_1''$ in $V_1$. Thus the second assertion  holds.

For $u'\otimes \mathbf{1}''\in V_1'\otimes \mathbf{1}'', \mathbf{1}'\otimes v''\in \mathbf{1}''\otimes V_1''$,
we have $$[u'\otimes \mathbf{1}'',\mathbf{1}'\otimes v'']= (u'\otimes \mathbf{1}'')_0(\mathbf{1}'\otimes v'')=u'_0\mathbf{1}'\otimes v''=a \mathbf{1}'\otimes v''.$$ 
Here $a\in \C$. Similarly,
$[\mathbf{1}'\otimes v'',u'\otimes \mathbf{1}'']=b u'\otimes \mathbf{1}''$ for some $b\in \C$.
The skew symmetric property of the Lie bracket implies $b u'\otimes \mathbf{1}''=-a\cdot\mathbf{1}'\otimes v''$. Thus $a=b=0$ and 
 $[u'\otimes \mathbf{1}'',\mathbf{1}'\otimes v'']=0$.
Hence $[V_1'\otimes \mathbf{1}'',\mathbf{1}'\otimes V_1'']=0.$ So the third assertion holds.

According to the definition of semi-conformal vectors, obviously, we have the fourth assertion.

For $\widetilde{\omega}'\in  \on{Sc}(V',\omega')$, from  \cite[ Corollary 3.11.11]{LL}, we have
$$
C_{V}(<\widetilde{\omega}'>\otimes \mbf{1}'')=\on{Ker}_{V'\otimes V''}(\widetilde{L}'(-1)\otimes \on{Id}'')=\on{Ker}_{V'}\widetilde{L}'(-1)\otimes V''.
$$
Since $\on{Ker}_{V'}\widetilde{L}'(-1)=C_{V'}(<\widetilde{\omega}'>)$, we get
$C_{V}(<\widetilde{\omega}'>\otimes \mbf{1}'')=C_{V'}(<\widetilde{\omega}'>)\otimes V''$.

Since $ \omega=\omega'\otimes \mbf{1}''+\mbf{1}'\otimes \omega''$, we have $<\widetilde{\omega}'>\otimes \mbf{1}''=<\widetilde{\omega}'\otimes \mbf{1}''>$. By  \cite[Corollary 3.11.11]{LL}, we have
$$C_{V}(C_{V}(<\widetilde{\omega}'>\otimes \mbf{1}''))=\on{Ker}_{V}(L(-1)-\widetilde{L}'(-1)\otimes \on{Id}'')=\on{Ker}(\on{Id}'\otimes L''(-1));$$
$$C_{V'}(C_{V'}(<\widetilde{\omega}'>))=\on{Ker}_{V'}(L'(-1)-\widetilde{L}'(-1)).$$
Since $L(-1)=L(-1)'\otimes\on{Id}''+\on{Id}'\otimes L''(-1)$, we have $ L(-1)=L(-1)'\otimes\on{Id}'' $ on $V'\otimes \mbf{1}''$ and  $$C_{V'}(C_{V'}(<\widetilde{\omega}'>))\otimes \mbf{1}''\subset C_{V}(C_{V}(<\widetilde{\omega}'\otimes \mbf{1}''>)).$$

Let $\eta=\sum_{i=1}^{t}\alpha_i\otimes \beta_i\in \on{Ker}_{V}(L(-1)-\widetilde{L}'(-1)\otimes \on{Id}'')$. Here
$\alpha_1,\cdots,\alpha_t$ are linearly independent homogeneous vectors 
in $V'$ and $\beta_1,\cdots,\beta_t$ are  homogeneous non-zero vectors in $V''$.  
We will use the fact that $ V_{n}=\oplus _{l\in \N}V'_{l}\otimes V''_{n-l}$, i.e.,  
$V$ is a $\Z\times \Z$-graded vector space and the operators $L'(-1)\otimes \on{Id}''$ and $\on{Id}'\otimes L''(-1) $ have bi-degrees $(1,0)$ and $(0,1)$ respectively. 

If $(L(-1)-\widetilde{L}'(-1)\otimes \on{Id}'')(\eta)=0$,~i,e.,
$L(-1)(\sum_{i=1}^{t}\alpha_i\otimes \beta_i)=\sum_{i=1}^{t}(\widetilde{L}'(-1)\alpha_i)\otimes \beta_i$,
then $$\sum_{i=1}^{t}L'(-1)\alpha_i\otimes \beta_i+\sum_{i=1}^{t}\alpha_i\otimes L''(-1)\beta_i=\sum_{i=1}^{t}\widetilde{L}'(-1)\alpha_i\otimes \beta_i \; \text{or equivalently }$$
$$\sum_{i=1}^{t}(L'(-1)-\tilde{L}'(-1))(\alpha_i)\otimes \beta_i=-\sum_{i=1}^{t}\alpha_i\otimes L''(-1)\beta_i.$$
We now claim that the only possible non-zero $\beta_i$ are in $ \C\mbf{1}''$. 
Since $V''$ is simple CFT type, $L''(-1)$ is injective on all $V''_n$ except 
$ n=0$ (see \cite{DLM1}). Suppose that 
$\{\beta_{i_1}, \cdots, \beta_{i_r}\}$ is the subset $\{\beta_1, \cdots, \beta_t\}$ 
of the same highest weight $n$. Then comparing both sides of 
the $(\Z, n+1)$ homogeneous components, we get $ \sum_{j=1}^r\alpha_{i_j}\otimes L''(-1)\beta_{i_j}=0$. Since $\alpha_i$'s are linearly independent, we have $L''(-1)\beta_{i_j}=0$. If $ n>0$, then $ \beta_{i_j}=0$. But we have assumed that $ \beta_i$'s are non-zero.  Hence $\beta_i \in \C\mbf{1}''$ and $ \eta=\alpha\otimes \mbf{1}'' \in V'\otimes \mbf{1}''$. Furthermore, $(L'(-1)-\tilde{L}'(-1))\alpha=0$, i.e., $ \alpha \in C_{V'}(C_{V'}(<\widetilde{\omega}'>)$.
Therefore, we have
$C_{V}(C_{V}(<\widetilde{\omega}'>)\subset C_{V'}(C_{V'}(<\widetilde{\omega}'>)\otimes \mbf{1}''$.
Thus, we get $C_{V}(C_{V}(<\widetilde{\omega}'>)= C_{V'}(C_{V'}(<\widetilde{\omega}'>)$. 
\end{proof}

\begin{lem} \label{lem5.2}
 Let $(V,\omega)$ be a nondegenerate simple CFT type vertex operator algebra generated by $V_1$.  If $V=V'\otimes V''$ and $V$ satisfies
\begin{equation}\label{f6.1}
\forall \tilde{\omega}\in \on{Sc}(V,\omega), V\simeq C_{V}(C_{V}(<\tilde{\omega}>))\otimes C_{V}(<\tilde{\omega}>),
\end{equation}
then $V',V''$ also satisfy (\ref{f6.1}).
Conversely, if $V',V''$ satisfy (\ref{f6.1}) and $\on{Sc}(V,\omega)=\on{Sc}(V',\omega')+\on{Sc}(V'',\omega'')$,
then $V$ also satisfies \eqref{f6.1}.
\end{lem}
\begin{proof}
Let $\omega',\omega''$ be the conformal vectros of $V'$ and $V''$, respectively. Since $V=V'\otimes V''$, then $\omega=\omega'\otimes \mbf{1}''+\mbf{1}'\otimes \omega''$ and $\omega',\omega''\in \on{Sc}(V,\omega)$. For more notational simplicity, we will omit the tensor factor $\otimes \mbf{1}''$ or $\mbf{1}'\otimes $ in the following argument. For $\forall \widetilde{\omega'}\in \on{Sc}(V',\omega')$, using 5) of Lemma \ref{l6.1},
we have
$$V=C_{V}(C_{V}(<\widetilde{\omega}'>))\otimes C_{V}(<\widetilde{\omega'}>)=C_{V'}(C_{V'}(<\widetilde{\omega'}>)
\otimes C_{V'}(<\widetilde{\omega'}>)\otimes V''.$$
This implies $V'=C_{V'}(C_{V'}(<\widetilde{\omega'}>))
\otimes C_{V'}(<\widetilde{\omega'}>)$. The same argument shows that
$V''=C_{V'}(C_{V''}(<\widetilde{\omega}''>))
\otimes C_{V''}(<\widetilde{\omega}''>)$ for $ \widetilde{\omega}''\in \on{Sc}(V'',\omega'')$.

Conversely, assume $V=V'\otimes V''$, $\on{Sc}(V,\omega)=\on{Sc}(V',\omega')+\on{Sc}(V'',\omega'')$, and
for each $\widetilde{\omega}'\in \on{Sc}(V',\omega'), \widetilde{\omega}''\in \on{Sc}(V'',\omega'')$, there are 
\begin{equation}\label{f6.2}
V'=C_{V'}(C_{V'}(<\widetilde{\omega'}>))\otimes C_{V'}(<\widetilde{\omega'}>);
\end{equation}
\begin{equation}\label{f6.3}
V''=C_{V'}(C_{V''}(<\widetilde{\omega}''>))
\otimes C_{V''}(<\widetilde{\omega}''>).
\end{equation}
Then we shall show that
\begin{equation}\label{f6.4}
V=C_{V}(C_{V}(<\widetilde{\omega}>))
\otimes C_{V}(<\widetilde{\omega}>)
\end{equation}
for any  $\widetilde{\omega}=\widetilde{\omega}'+\widetilde{\omega}''\in \on{Sc}(V,\omega)$.
Since (\ref{f6.2}), (\ref{f6.3}) hold, then we have
$$V=
C_{V'}(C_{V'}(<\widetilde{\omega}'>)\otimes C_{V'}(<\widetilde{\omega}'>)\otimes C_{V''}(C_{V''}(<\widetilde{\omega}''>))
\otimes C_{V''}(<\widetilde{\omega}''>).$$

We first claim that $C_{V'}(<\widetilde{\omega}'>)\otimes C_{V''}(<\widetilde{\omega}''>)=C_{V}(<\widetilde{\omega}>).$

In fact, $\forall \alpha\otimes \beta\in C_{V'}(<\widetilde{\omega'}>)\otimes C_{V''}(<\widetilde{\omega}''>)$
for $\alpha\in C_{V'}(<\widetilde{\omega}'>)$ and $\beta\in C_{V''}(<\widetilde{\omega}''>)$, we have
$$\widetilde{L}(-1)(\alpha\otimes \beta)=(\widetilde{L}'(-1)+\widetilde{L}''(-1))(\alpha\otimes \beta)=\widetilde{L}'(-1)\alpha\otimes \beta+\alpha\otimes
\widetilde{L}''(-1)\beta=0.$$ So  $C_{V'}(<\widetilde{\omega}'>)\otimes C_{V''}(<\widetilde{\omega}''>)\subset C_{V}(<\widetilde{\omega}>).$

Similar to the proof of Lemma~\ref{l6.1} 5), we set  $\eta=\sum_{i=1}^{t}\alpha_i\otimes \beta_i\in C_{V}(<\widetilde{\omega}>)$, where
$\alpha_1,\cdots,\alpha_t$ are homogeneous  nonzero vectors in $V'$ and $\beta_1,\cdots,\beta_t$ are homogeneous non-zero vectors in $V''$. Then
\begin{equation}\label{eq5.5} (\widetilde{L}'(-1)+\widetilde{L}''(-1))(\sum_{i=1}^{t}\alpha_i\otimes \beta_i)=
\sum_{i=1}^{t}\widetilde{L}'(-1)\alpha_i\otimes \beta_i+\sum_{i=1}^{t}\alpha_i\otimes \widetilde{L}''(-1)
\beta_i=0.
\end{equation}
Unlike the proof of  Lemma~\ref{l6.1} 5), one cannot use the injectivity of $\tilde{L}'(-1) $ on $V'_n$ with $n>0$ (and similar property of $L''(-1)$).  
We still use the $\Z\times \Z$-gradation on $V'\otimes V''$.  
We first assume that there exists $n, m \geq 0 $ such that $\on{wt}(\alpha_i)=m$ and 
$\on{wt}(\beta_i)=n$ for all $i=1, \cdots, t$. Then the first summation in \eqref{eq5.5} has degree $(m+1, n)$ while the second   summation in  \eqref{eq5.5} has degree $(m, n+1)$. Thus we have 
$$ \sum_{i=1}^{t}\widetilde{L}'(-1)\alpha_i\otimes \beta_i=0 \quad \text{ and} \quad\sum_{i=1}^{t}\alpha_i\otimes \widetilde{L}''(-1)
\beta_i =0.$$
By assuming that $\alpha_i$'s are linearly independent (without changing $\eta$), we get that $ \tilde{L}''(-1)\beta_i=0$ and thus $\eta= \sum_{i=}^t \alpha_i\otimes \beta_i \in V'_m\otimes \on{Ker}(\tilde{L}''(-1): V''_n\rightarrow V''_{n+1})$. Similarly, we can take $\{ \beta_1, \cdots, \beta_t\}$ to be linearly independent (still without changing $\eta$), we can get $\tilde{L}'(-1)\alpha_i=0$ for all $ i$ and thus $ \eta \in \on{Ker}(\tilde{L}'(-1): V'_m\rightarrow V'_{m+1})\otimes V''_{n}$.  

Since $ (V'_{m+1}\otimes V''_{n})\oplus (V'_{n}\otimes V''_{m+1}) \subseteq (V'\otimes V'')_{m+n+1}$ is a direct sum, and a standard linear algebra argument shows that 
$$(\on{Ker}(\tilde{L}'(-1)))_{m}\otimes (\on{Ker}(\tilde{L}''(-1)))_n=(\on{Ker}(\tilde{L}'(-1)))_{m}\otimes V''_n)\cap (V'_{m}\otimes (\on{Ker}(\tilde{L}''(-1)))_{n}). $$
Hence $ \eta\in  (\on{Ker}(\tilde{L}'(-1)))_{m}\otimes (\on{Ker}(\tilde{L}''(-1)))_n$. 

For general $ \eta$, let $(m,n)$ be a maximal element of 
$\{ (\on{wt}\alpha_i, \on{wt}{\beta_i})\;|\; i=1, \cdots, t\}$ with 
respect to the partial order: $(m,n)\geq (k,l) $ iff $ m\geq k$ and 
$ n\geq l$. Let $\eta_{m,n}$ be the $(m,n)$ homogeneous component 
of $ \eta$.   Comparing the bi-degrees in \eqref{eq5.5}, we see that 
$\tilde{L}(-1)\eta=0$ implies that $\tilde{L}(-1)(\eta_{m,n})=0$. 
The special case we discussed above shows that $$ \eta_{m,n} \in  (\on{Ker}(\tilde{L}'(-1)))_{m}\otimes (\on{Ker}(\tilde{L}''(-1)))_n\subseteq C_{V'}(<\widetilde{\omega}'>)\otimes C_{V''}(<\widetilde{\omega}''>).$$

Now one replaces $\eta$ by $ \eta-\eta_{m,n}$ and recursively show that 
$ \eta \in C_{V'}(<\widetilde{\omega}'>)\otimes C_{V''}(<\widetilde{\omega}''>)$. Therefore, we have just proved $ C_V(<\tilde{\omega}>)\subseteq C_{V'}(<\widetilde{\omega}'>)\otimes C_{V''}(<\widetilde{\omega}''>)$. 

Thus, we have just proved the claim $ C_V(<\tilde{\omega}>)= C_{V'}(<\widetilde{\omega}'>)\otimes C_{V''}(<\widetilde{\omega}''>)$.

By considering $ \omega-\tilde{\omega}$ and using the fact the $C_V(C_V(<\omega'>))=C_{V}(<\omega-\omega'>)$,  we can get 
$$
C_{V'}(C_{V'}(<\widetilde{\omega}'>))\otimes C_{V''}(C_{V''}(<\widetilde{\omega}''>))=C_{V}(C_{V}(<\widetilde{\omega}>)).$$
Now by substituting these in the formula \eqref{f6.4} and using the assumptions \eqref{f6.2} and \eqref{f6.3}, we can show that the formula (\ref{f6.4}) holds.
\end{proof}

\begin{lem}\label{lem5.3}
If $(V,\omega)$ is a nondegenerate simple CFT type vertex operator algebra generated by $V_1$ satisfying (\ref{f6.1}) and  $\widetilde{\omega} \in \on{Sc}(V,\omega)$ is neither $0$ nor $\omega$, then
$$\on{dim} C_{V}(C_{V}(<\widetilde{\omega}>))_1>0\; \;\text{and}\;\; \on{dim} C_{V}(<\widetilde{\omega}>)_1>0.$$
\end{lem}
\begin{proof}
Since  the vertex operator algebra $(V,\omega)$ satisfies (\ref{f6.1}), we have
$$\on{dim} C_{V}(C_{V}(<\widetilde{\omega}>))_1+\on{dim} C_{V}(<\widetilde{\omega}>)_1=\on{dim} V_1.$$
Note that $\widetilde{\omega} \in \on{Sc}(V,\omega)$ is  not zero, then $ C_V(<\tilde{\omega}>)\neq V$. Hence $ C_V(<\tilde{\omega}>)_1\neq V_1$. And since $V$ is generated by $V_1$, then $ C_V(C_V(<\tilde{\omega}>))_1 \neq 0$.  Since $ \tilde{\omega}\neq \omega$, then $ \omega-\tilde{\omega}\neq 0$. Thus   $\on{dim} C_{V}(C_{V}(<\widetilde{\omega}>))_1=\on{dim} C_{V}(<\omega-\widetilde{\omega}>)_1> 0 $. So $\on{dim} C_{V}(C_{V}(<\widetilde{\omega}>))_1>0$ and $\on{dim} C_{V}(<\widetilde{\omega}>)_1>0$.
\end{proof}
\begin{prop} \label{p5.4}
Let $(V,\omega)$ be a nondegenerate simple CFT type vertex operator algebra generated by $V_1$.
If $h\in V_1$ such that $ \langle h, h\rangle =1 $ and $ L(1)h=0$, then the vertex subalgebra $<h>$ of $V$ is isomorphic to $ V_{\hat{\h}}(1,0)$ and the pair $(<h>, \omega')$ with $\omega'=\frac{1}{2}h_{-1}h_{-1}\mbf{1}\in V_2$, is a semi-conformal subalgebra of $(V, \omega)$.  In particular,  if $\on{dim}V_1=1$ such that $ L(1)V_1=0$, then $(V,\omega)\cong (V_{\hat{\h}}(1, 0),\omega_\Lambda)$ with $ \Lambda=0$.
\end{prop}
\begin{proof}
Let $h \in V_1$ be as in the assumption with $h_1 h=\langle h,h\rangle \mbf{1}=\mbf{1}$. Since $V$ is $ \N-$ graded, we have $h_nh=0$ for all $n\geq 2$ and 
\begin{equation} \label{f5.6}[h_m,h_n]=\sum_{k=0}^{+\infty}
\left(
\begin{array}{lll}
m\\
k
\end{array}
\right)(h_kh)_{m+n-k}
=(h_0h)_{m+n}+m(h_1h)_{m+n-1},
\end{equation}
for $m,n\in \Z$.  Since $[v,u]=v_0(u)$ defines a Lie algebra structure on $V_1$, hence $ h_0h=[h,h]=0$.     So we have
$$[h_m,h_n]=m\langle h,h\rangle \delta_{m+n,0}\on{Id}=m\delta_{m+n,0}\on{Id}.$$
Let $\h=\C h$. This defines  a Heisenberg  Lie algebra homomorphism $\hat{\h}\rightarrow \on{End}(V)$ with $ h\otimes t^n\mapsto h_n$ and $C\mapsto \on{Id}$.(cf. Section \ref{sec3.1}).  Let $ U=<h>$ be the vertex subalgebra generated by $h$. Then  there is a vertex algebra
 homomorphism $V_{\hat{\h}}(1,0) \rightarrow V$ sending $ h(-1)\mbf{1}$ to $ h_{-1}\mbf{1}=h$. The mage of this map is exactly $ U$. Since $V_{\hat{\h}}(1,0)$ is a simple vertex algebra. Thus $U\cong V_{\hat{\h}}(1,0)$.  In the following we discuss the conformal elements. The vertex algebra $V_{\hat{\h}}(1,0)$ with the conformal vector $ \omega'=\frac{1}{2}h(-1)h(-1)\mbf{1}$  and the central charge $c=1$. 
 Its image in $U$ will still be denoted by $ \omega'$ and $ \omega'=\frac{1}{2} h_{-1}h_{-1}\mbf{1} \in V_2$. Thus $(U, \omega')$ is a vertex operator subalgebra of $(V, \omega)$.  
 We now show that this map is semi-conformal or and  $ \omega'$ is semi-conformal in $(V, \omega)$
if $ L(1)h=0$ (the condition has not yet been used!).  By using Proposition~\ref{p2.2}, we only need to check the following equations, since $(U, \omega')$ is already  a vertex operator algebra. 
\begin{align}\label{5.7}\left\{\begin{array}{llll}
L(0)\omega'=2\omega';\\
L(1)\omega'=0;\\
L(2)\omega'=\frac{1}{2}\mbf{1};\\
L'(-1)\omega'=L(-1)\omega';\\
L(n)\omega'=0, n\geq 3.
\end{array}\right.
\end{align}
 The first  is true since $ \omega'\in V_2$. The last one holds since $ V$ is $\N$-graded. The fourth one holds always.  Similar to \eqref{f5.6}, using the assumption that $ L(1)h=0$ and the fact $L(k)h=0$ for $k\geq 2 $, we have 
 \begin{equation} \label{f5.8}[L(m),h_n]=\sum_{k=-1}^{+\infty}
\binom{m+1}{k+1}(L(k)h)_{m+n-k}
=(L(-1)h)_{m+n+1}+(m+1)(L(0)h)_{m+n}.
\end{equation}
Using the fact that $ Y(L(-1)v, z)=\frac{d}{dz} Y(v, z)$ one gets $ (L(-1)v)_{n+1}=-(n+1)v_{n}$. Hence 
\[ [L(m),h_n]=-nh_{n+m}\]
for all $ m, n\in  \Z$. 
Thus 
\begin{eqnarray*}
L(2)\omega'&=&\frac{1}{2}[L(2), h_{-1}^2]\mbf{1}\\
&=& \frac{1}{2}([L(2),h_{-1}]h_{-1}+h_{-1}[L(2), h_{-1}])\mbf{1}\\
&=& \frac{1}{2}(h_{1}h_{-1}+h_{-1}h_{1})\mbf{1}=\frac{1}{2}\mbf{1}.
\end{eqnarray*}
\begin{eqnarray*}
L(1)\omega'&=&\frac{1}{2}[L(1), h_{-1}^2]\mbf{1}\\
&=& \frac{1}{2}([L(1),h_{-1}]h_{-1}+h_{-1}[L(1), h_{-1}])\mbf{1}\\
&=& \frac{1}{2}(h_{0}h_{-1}+h_{-1}h_{0})\mbf{1}=\frac{1}{2}h_{0}h_{-1}\mbf{1}=\frac{1}{2}h_{0}h=0.
\end{eqnarray*}
 
  Now assume $ \dim V_1=1$.  Then any orthonormal element $h \in V_1$ satisfies the condition $L(1)h=0$. Since $V$ is generated by $V_1$, hence $V\simeq <h>=V_{\widehat{V_1}}(1,0)$ and the conformal vector of $V$ is $\frac{1}{2}h_{-1}^2\mbf{1}$.  
 \end{proof}
 We remark that a consequence of this proposition is the following which will be repeatedly used in the proof of Theorem~\ref{thm1.4}.
 
 1) Under the assumption that $ L(1)V_1=0$, than any $ h \in V_1$ with the property that $\langle h, h\rangle =1$ defines an element $ \omega_h\in \on{Sc}(V, \omega)$;

2) Note that $ C_V(\{h\})=C_V(<h>)=C_V(C_V(C_V(<h>)))$. If  $ C_{V_1}(h)=\{ v\in V_1 \; |\; h_0(v)=0\} $ is the centralizer of $h$ in the Lie algebra $V_1$ and $ h^\perp =\{ v\in V_1 \; |\; \langle v, h\rangle =0\} $ is the subspace, then $ C_{V_1}(h)\cap h^\perp \subseteq (C_V(<h>))_1$.  In particular, let  $V_1$ be an abelian Lie algebra and $ \dim V_1\geq 2$. Then $ \on{Sc}(V, \omega)\setminus \{0, \omega\}$ is not empty if $\on{Ker}(L(1): V_1\rightarrow \C \mbf{1})$ contains an element $h$ with $ \langle h, h \rangle \neq 0$.  

3) If $ (V, \omega)$ has central charge $\neq 1$, then for any $ h\in V_1$ as in Proposition~\ref{p5.4}, then $ C_V(C_V(<h>))\neq V$. Hence, $ \on{Sc}(V, \omega)\setminus \{0, \omega\}$ is not empty. 

4) For any $ \omega'\in \on{Sc}(V, \omega)$, if $ L(1)V_1=0$, then $ L'(1)(C_V(C_V(<\omega'>))_1)=0$.
 
\begin{proofof}{\bf Proof of Theorem 1.4.} We will use induction on $\dim V_1$.  If $\on{dim}V_1=1$, the theorem is  the special case in Proposition~\ref{p5.4}. Hence $(V,\omega)\cong V_{\widehat{V_1}}(1,0)$ with central charge $1$.

Assume  $\on{dim}V_1\geq 2$. Since the bilinear form $ \langle\cdot, \cdot \rangle $ on $ V_1$ is non-degenerate, there is an $ h\in V_1$ such that $\langle h, h \rangle=1$. Since $ L(1)V_1=0$, Proposition~\ref{p5.4} implies that $\omega'=\omega_h=\frac{1}{2} h_{-1}^2\mbf{1}\in V_2$ is a semi-conformal vector of $(V,\omega)$. 
Then the assumption of the theorem implies 
\begin{equation} 
\label{f5.9} (V, \omega)\cong (C_V(<\omega_h>), \omega-\omega_h)\otimes (C_V(<\omega-\omega_h>), \omega_h)
\end{equation}
as vertex operator algebras. 

Let $ V'(h)=C_V(<\omega-\omega_h>)$ with conformal vector $\omega_h$ and $ V''(h)=C_V(<\omega_h>)$ with conformal vector $ \omega-\omega_h$.  Then, by Lemma~\ref{l6.1}. 1), both $(V'(h), \omega_h)$ and $(V'(h), \omega-\omega_h)$ satisfy the conditions of this theorem. If we can prove that  
\begin{equation}\label{f5.10}
\dim V'(h)_1<\dim V_1\quad \text{ and } \quad \dim V''(h)_1<\dim V_1,
\end{equation}
 then the induction assumption will imply that both $V'(h)$ and $V''(h)$ are Heisenberg vertex operator algebras of the specified rank and of the specified conformal vector. Thus the tensor product vertex operator algebra also has the same property.   By Lemma~\ref{lem5.3}, we only need to show that $ \omega_h\neq \omega $ when $ \dim V_1 \geq 2$. However one cannot be sure that  $ \omega_h\neq \omega$ in this case for arbitrarily chosen $ h\in V_1$.  But we can choose  $h\in V_1$ with $ \langle h, h\rangle =1$ such that $ \dim V'(h)_1$ is the smallest possible. We claim that $ \dim V'(h)_1=1$ and then $ \dim V''(h)_1=\dim V_1-1>0$. Then we are done. In the following we prove this claim. 
 
 Note that $ h\in V'(h)_1$ implies that $ \dim V'(h)_1>0$. Assume that $ \dim V'(h)_1\geq 2$. By Lemma~\ref{lem5.1}.1),  $ V'(h)$ is nondegenerate simple CFT type and is generated by $V'(h)_1$ as a vertex algebras. Thus the bilinear form $\langle\cdot, 
 \cdot \rangle$ remains nondegenerate when restricted to $ V'(h)_1$. There is an $ h'\in 
 V'(h)_1$ such that $ \langle h', h'\rangle =1$ and $ \langle h, h'\rangle =0$. Hence $h$ and $ h'$ are linearly independent. Set $ \omega_{h'}=\frac{1}{2}h'_{-1}h'_{-1}\mbf{1} \in V'(h)_2$. Then $\omega_{h'}$ is a semi-conformal vector of $V$ by Proposition~\ref{p5.4}. By using Proposition~\ref{p2.2}, $\omega_{h'}$ is also a semi-conformal vector of $V'(h)$. Thus $ \omega_{h'}\preceq \omega_{h}$. 
 We now claim that $\omega_{h'}\neq \omega_{h}$. Applying Lemma~\ref{lem5.3}, we will have $\dim C_{V'(h)}(C_{V'(h)}(<\omega_{h'}>))_1<\dim V'(h)_1$. 
Since  $ C_{V'(h)}(C_{V'(h)}(<\omega_{h'}>))=C_V(C_V(<\omega_{h'}>))=V'(h')$, 
we have a contradiction to the choice of $h$. 
Hence we must have $ \omega_h=\omega_{h'}$. Note that $V'(h)$ contains 
to vertex subalgebras $U(h)$ and $U(h')$ generated by $ h$ and $h'$ 
respectively with $ \omega_h$ and $\omega_{h'}$ being conformal 
vectors respectively as in the proof of Proposition~\ref{p5.4}  and $ h'_{1}\omega_{h'}=h'$ and $ h_{1}\omega_{h}=h$. But using 
\[ [h'_1, h_{-1}]=(h'_0h)_0+(h'_1h)_{-1}=(h'_0h)_0 \quad \text{and } \quad v_{n}\mbf{1}=0 \; \text{ for all } n\geq 0, \]
we compute 
\begin{eqnarray*}
h'&=&h'_{1}\omega_{h'}=h'_1\omega_{h}=\frac{1}{2}( [h'_1, h_{-1}]h_{-1}+h_{-1}[h'_1,h_{-1}])\mbf{1}\\
&=& \frac{1}{2}(h'_0h)_0 h_{-1}\mbf{1}=\frac{1}{2}[(h'_0h)_0, h_{-1}]\mbf{1}\\
&=&\frac{1}{2} ((h'_{0}h)_0h)_{-1}\mbf{1}=\frac{1}{2}((h'_{0}h)_0h=\frac{1}{2}[[h', h],h].
\end{eqnarray*}
 By setting $ h''=[h', h]$,  we have $ [h'',h]=2h'$ in the Lie algebra $ V_1$. 
 Exchanging $ h$ and $h'$ we also get $ [[h, h'], h']=2h$, i.e., $ [h'', h']=-2h$.   This also implies that $ h''\neq 0$ and  the linear span of   $\{h, h', h''\}$ is Lie algebra isomorphic to $ \frak{sl}_2$ with  the standard generators 
 \begin{equation} e= \frac{h+ih'}{\sqrt{2i}} \quad, f= \frac{ih+h'}{\sqrt{2i}}, \quad k=ih''.
  \end{equation}
  Thus we have 
  \[[ k,e]=2e, \quad [k,f]-2f, \quad [e,f]=k.\]
  Since the bilinear form $ \langle \cdot, \cdot \rangle $ is invariant, a direct computation shows that 
  \[\langle h'', h\rangle=\langle h'', h'\rangle =0 \quad \text{and} \quad \langle h'', h''\rangle=-2.\] 
  Then $ \omega_{h''/\sqrt{-2}}$ is also a semi-conformal vector of $ V'(h)$. If $ \omega_{h''/\sqrt{-2}}\neq \omega_h$, we are done. We shall show that this is the case. 
  
  Let $ U$ be the vertex subalgebra generated by $ \{h, h', h''\}$ in $ V'(h)$ with the conformal vector $ \omega_h$ with has central charge 1. Then $U$ is a quotient of the affine vertex algebra $ V_{\widehat{\frak{sl}}_2}(1,0)$ and there is a surjective map $ U\rightarrow L_{\widehat{\frak{sl}}_2}(1,0)$ such that the images of $ \omega_{h}$, $ \omega_{h'}$, and $\omega_{h''}$ in $ L_{\widehat{\frak{sl}_2}}(1,0)$ are 
  \[ \tilde{\omega}_{h}=\frac{1}{2}h(-1)h(-1)\mbf{1}, \quad  \tilde{\omega}_{h'}=\frac{1}{2}h'(-1)h'(-1)\mbf{1}, \quad \text{and} \quad  \tilde{\omega}_{h''}=\frac{1}{2}h''(-1)h''(-1)\mbf{1}\]
  respectively. Here we are using the notation $ h(-1)=h\otimes t^{-1}$ in the affine Lie algebra. Note that, under the basis $ \{e, f, k\}$ in $ \frak{sl}_2$, we have 
  \[ h=\frac{\sqrt{2}i}{2}(e-if), \quad  h'=\frac{\sqrt{2}i}{2}(f-ie), \quad \text{ and } \quad h''=-ik.\]
 
  Expressing three vectors $\omega_h,\omega_h',\omega_h''$  in terms of $ \{e, f, k\}$ in $L_{\widehat{\frak{sl}}_2}(1,0)=V_{\widehat{\frak{sl}}_2}(1,0)/<e(-1)^2\mbf{1}> $,  we have 
  \begin{eqnarray*}
\tilde{\omega}_{h}&=& \frac{1}{4}(f(-1)f(-1)+if(-1)e(-1)+ie(-1)f(-1))\mbf{1},\\
\tilde{\omega}_{h'}&=& \frac{1}{4}(-f(-1)f(-1)+if(-1)e(-1)+ie(-1)f(-1))\mbf{1},\\
 \tilde{\omega}_{h''\sqrt{-2}}&=&\frac{1}{4}k(-1)k(-1)\mbf{1}.
  \end{eqnarray*}
 Hence $ \tilde{\omega}_{h}-\tilde{\omega}_{h'}=\frac{1}{2}f(-1)^2\mbf{1}\neq 0$ by considering a PBW basis in $V_{\widehat{\frak{sl}}_2}(1,0)$ in the order $ \prod_r f(-n_{r})\prod_s k(-m_{s})\prod_t e(-l_{t}) $ with all $ n_r, m_s, l_t>0$ and using the fact that the weight two subspace of $<e(-1)^2\mbf{1}>$ is $ e(-1)^2\mbf{1}$. This also implies that  $ \tilde{\omega}_h$,  $ \tilde{\omega}_{h'},$ and $ \tilde{\omega}_{h''/\sqrt{-2}}$ are distinct in $ L_{\widehat{\frak{sl}}_2}(1,0)$.   This contradicts the early assumption that   $\omega_{h}=\omega_{h'}={\omega}_{h''/\sqrt{-2}}$.  Thus, we  have completed the proof of the theorem. \qed
  \end{proofof}   

\begin{remark}
1) The condition $ L(1)V_1=0$ is satisfied for most vertex operator algebras arising from Lie algebras with standard conformal structures \cite[Section 7.1.1, 7.1.2]{BF}. 

2) In Theorem 1.4,   the condition $ L(1)V_1=0$ is crucial in distinguishing the standard conformal structure $\omega_{\Lambda}$ with $ \Lambda=0$ on $ V_{\hat{\h}}(1,0)$ 
from other $ \omega_{\Lambda}$ with $ \Lambda\neq 0$.

3)   The tensor decomposition condition \eqref{e5.5}
\begin{equation*}
\forall \omega'\in \on{Sc}(V,\omega), V\simeq C_{V}(C_{V}(<\omega'>))\otimes C_{V}(<\omega'>)
\end{equation*}   is    critical in conditions to distinguishing Heisenberg vertex algebra from other vertex algebras arising from Lie algebras.  Checking this condition could be a challenge.  

 For a vertex operator algebra $(V, \omega)$ of  nondegenerated simple CFT type, by Lemma 5.1,  the condition \eqref{e5.5} implies that for an $\omega' \in\on{Sc}(V, \omega)$,   the following assertions hold:
\begin{itemize}
\item[(1)]
$C_{V}(C_{V}(<\omega'>))_1$ and $C_{V}(<\omega'>)_1$ are mutually orthogonal in $V_1$;

\item[(2)] $V_1\simeq C_{V}(C_{V}(<\omega'>))_1\oplus C_{V}(<\omega'>)_1.$
\end{itemize}

\end{remark}
On the other hand, the following Lemma will provide a way to verify the condition \eqref{e5.5}. It will also be used in the proof of Theorem~\ref{thm1.5}.

\begin{lem} \label{lem5.7} Let $(V,\omega)$ be a nondegenerate simple CFT type vertex operator algebra.
For $\omega'\in \on{Sc}(V,\omega)$, assume  $\on{dim}C_{V}(C_{V}(<\omega'>))_1\neq 0$, $\on{dim}C_{V}(<\omega'>)_1\neq 0$,  and $\dim V_1=\dim C_{V}(<\omega'>)_1 +\dim C_{V}(C_{V}(<\omega'>))_1$, then we have
\begin{itemize}
\item[1)] $V=C_{V}(<\omega'>)\otimes C_{V}(C_{V}(<\omega'>));$
\item[2)] $C_{V}(<\omega'>)=<C_{V}(<\omega'>)_1> \text{ and} \; C_{V}(C_{V}(<\omega'>))=<C_{V}(C_{V}(<\omega'>))_1>.$
\end{itemize}
\end{lem}
\begin{proof}
Note that $<C_{V}(<\omega'>)_1>\otimes<C_{V}(C_{V}(<\omega'>))_1>$ is a vertex operator subalgebra of $V$ with the same conformal vector  $\omega$ and weight one subspace being $ C_{V}(<\omega'>)_1\oplus C_{V}(C_{V}(<\omega'>))_1$. The assumption of the dimensions of weight one subspaces implies 
 $$C_{V}(<\omega'>)_1\oplus C_{V}(C_{V}(<\omega'>))_1=V_1.$$
Since $ V$ is generated by $ V_1 $ and the centres of $C_{V}(<\omega-\omega'>)$ and $ C_{V}(<\omega'>)$ are all one dimensional subspace $ \C\mbf{1}$,   then the assertions 1) and 2) will follow from the definition and Lemma~\ref{l6.1}. 
\end{proof}

\begin{proofof}{\bf Proof of Theorem 1.5.}
For notational convenience, we denote $U(\omega')=C_{V}(C_{V}(<\omega'>))$ for each $\omega'\in \on{Sc}(V,\omega)$.  Then $U(\omega-\omega')=C_{V}(<\omega'>)$.
 Note that $U(\omega^1)\otimes U(\omega^2-\omega^1)\otimes\cdots \otimes U(\omega^d-\omega^{d-1})$ is a conformal vertex operator subalgebra  of $V$ with the same conformal vector $\omega$. We know that
$U(\omega^1)_1\oplus U(\omega^2-\omega^1)_1\oplus \cdots \oplus U(\omega^d-\omega^{d-1})_1$ is a subspace of $V_1$. Since $\on{dim}U(\omega^{i}-\omega^{i-1})_1\neq 0$ for $i=1,\cdots,d$, then $\on{dim}U(\omega^{i}-\omega^{i-1})_1\geq 1$. Hence
$$V_1=U(\omega^1)_1\oplus U(\omega^2-\omega^1)_1\oplus \cdots \oplus U(\omega^d-\omega^{d-1})_1.$$  Thus $ \dim U(\omega^i-\omega^{i-1})_1=1$.  By  the given chain condition  and an induction on $d$  using  Lemma~\ref{lem5.7}, we have  $$V=<V_1>=<U(\omega^1)_1>\otimes <U(\omega^2-\omega^1)_1>\otimes\cdots \otimes <U(\omega^d-\omega^{d-1})_1>
.$$  Thus $V=U(\omega^1)\otimes U(\omega^2-\omega^1)\otimes\cdots \otimes U(\omega^d-\omega^{d-1})$ and  $U(\omega^{i}-\omega^{i-1})=<U(\omega^{i}-\omega^{i-1})_1>$ for $i=1,\cdots,d$. Since  $\on{dim}U(\omega^{i}-\omega^{i-1})_1=1$ and the assumption $L(-1)V_1=0$, Proposition~\ref{p5.4} implies that   $U(\omega^{i}-\omega^{i-1})$ is isomorphic to a Heisenberg vertex operator algebra $V_{\hat{\h}}(1,0)$ with $\dim \h =1$ and conformal vector of the form $ \omega _{\Lambda}$ with $ \Lambda=0$.  Therefore, $V$ is isomorphic to the Heisenberg vertex operator algebra $V_{\widehat{\h}}(1,0)$ with $ \h=V_1$ and the conformal vector of the form $ \omega_{\Lambda}$ with $ \Lambda =0$. \qed
\end{proofof}
\begin{remark} Theorem~\ref{thm1.5} can also be proved using the argument used in the proof of  Theorem~\ref{thm1.4}. The advantage of Theorem~\ref{thm1.5} is that one does not have to verify the condition \eqref{e5.5} for all $ \omega' \in \on{Sc}(V,\omega)$.  The condition of Theorem~\ref{thm1.5} is more convenient to verify when one uses a specific construction of $V$ such as most of the known examples.  Once again, the condition $ L(1)V_1=0$ is to ensure that the condition $ \Lambda=0$ for the conformal vector $\omega_{\Lambda}$. When the condition $ L(1)V_1=0$ is removed, one expects that $V$ is still isomorphic to $ V_{\hat{\h}}(1, 0)$ with conformal vector $ \omega_{\Lambda}$ and $ \Lambda$ needs not be zero. More discussions on general Heisenberg vertex operator algebras  will appear in an upcoming paper. 
\end{remark}
\section{The conclusion and further discussion}
This paper intends to understand properties of vertex operator algebras in terms of a certain geometric objects attached to them. For a general vertex operator algebra $(V,\omega)$, Theorem 1.1 tells us the set $\on{Sc}(V,\omega)$ of semi-conformal vectors of $(V,\omega)$ is an affine algebraic variety and it has a partial ordering (Definition 2.7) and an involution map induced by the coset construction.

To  describe the structure of the variety $\on{Sc}(V,\omega)$  for a general vertex operator algebra $(V,\omega)$ and to see how this variety determines  the vertex operator algebra $(V,\omega)$,  Heisenberg vertex operator algebras $(V_{\widehat{\h}}(1,0),\omega_{\Lambda})$ for $\Lambda\in \C^{d}$ are the first examples to look at. In fact, different $\Lambda$'s result in non-isomorphic conformal structures on the same Heisenberg vertex algebra $V_{\widehat{\h}}(1,0)$. In Section 3, we concentrate on the case $\Lambda=0$ and then described the orbit structure of $\on{Sc}(V_{\widehat{\h}}(1,0),\omega)$  under the action of the automorphic group $\on{Aut}(V_{\widehat{\h}}(1,0),\omega)$(Theorem 1.2).
Moreover, we study the maximal semi-conformal subalgebras determined by semi-conformal vectors. These subalgebrs are all Heisenberg-type vertex operator algebras and the Heisenberg vertex operator algebra $(V_{\widehat{\h}}(1,0),\omega)$ can be decomposed into the tensor product of any maximal semi-conformal subalgebra and its commutant in $(V_{\widehat{\h}}(1,0),\omega)$ (Theorem1.3). Moreover, based on properties of affine algebraic variety $\on{Sc}(V_{\widehat{\h}}(1,0),\omega)$,  for a nondegenerate simple CFT-type vertex operator algebra $(V,\omega)$ generated by $V_1$, we proved $(V,\omega)$ is  isomorphic to Heisenberg vertex operator algebra $(V_{\widehat{V_1}}(1,0),\omega)$ with the assumption of the condition $L(1)V_1=0$(Theorem1.4 and Theorem 1.5). Finally, we give two characterizations of $(V_{\widehat{\h}}(1,0),\omega)$ by its variety of semi-conformal vectors. 

For the Heisenberg vertex algebra $V_{\widehat{V_1}}(1,0)$, different $\Lambda'$s determine different conformal vectors $\omega_{\Lambda}$, and then  give non-isomorphic Heisenberg vertex operator algebras  $(V_{\widehat{\h}}(1,0),\omega_{\Lambda})$. For $\Lambda\neq 0$, the automorphic group of Heisenberg vertex operator algebra  $(V_{\widehat{\h}}(1,0),\omega_{\Lambda})$ will change and thus determining the orbit structure of their varieties of semi-conformal vectors will require much more work than that of the case $\Lambda= 0$. That will be an upcoming  work. Furthermore, for Heisenberg vertex operator algebras  $(V_{\widehat{\h}}(1,0),\omega_{\Lambda})$ with $\Lambda\neq 0$, we will work for  characterizations of them. As we have  mentioned in Remark 5.5, the condition $L(1)V_1=0$ is critical in distinguishing conformal structure $\omega_{\Lambda}$ with $\Lambda= 0$ on $V_{\widehat{\h}}(1,0)$ from other $\omega_{\Lambda}$ with $\Lambda\neq 0$, since Heisenberg vertex operator algebras $(V_{\widehat{\h}}(1,0),\omega_{\Lambda})$ for $\Lambda\neq 0$ will not satisfy the condition $L(1)V_1=0$.  Characterizing  Heisenberg vertex operator algebras  $(V_{\widehat{\h}}(1,0),\omega_{\Lambda})$ with $\Lambda\neq 0$ shall be more difficult  than   that of the case $\Lambda= 0$ and but a more meaningful problem.

  For each $\omega'\in \on{Sc}(V_{\widehat{\h}}(1,0),\omega)$, we want to
describe the set of all semi-conformal subalgebras with   $\omega'$ being the conformal vector.  Each of such semi-conformal subalgebras is a  conformal extension of $<\omega'>$ in $V_{\widehat{\h}}(1,0)$.  We  denote this set  by $\pi^{-1}(\omega')$ which is exactly the set of all conformal subalgebras of the smaller Heisenberg vertex operator algebra $V_{\widehat{\on{Im}\mathcal{A}_{\omega'}}}(1,0)$.  It is an interesting question to determine this set.  This depends on  the $<\omega'>$-module structure of  $V_{\widehat{\on{Im}\mathcal{A}_{\omega'}}}(1,0)$.   For affine vertex operator algebras such that the conformal subalgebra $<\omega'>$ is rational, then decomposing  each of the members in $\pi^{-1}(\omega')$ as direct sums of irreducible modules was the motivation for the study of semi-conformal subalgebras.

Similar to the Heisenberg case, same question can also be asked   for an affine vertex operator algebra $V$ corresponding to a finite dimensional simple Lie algebra $ \frak{g}$,  the variety $\on{Sc}(V, \omega)$ is closely related  to the  Lie algebra $\frak{g}$ with the fixed level $k$. In the case that $(V, \omega)$ is a lattice vertex operator algebra corresponding to a lattice $ (L, \langle\cdot, \cdot \rangle )$, the variety $\on{Sc}(V, \omega)$ is determined completely by the lattice structure.   In the upcoming work,  we shall  also describe invariants of affine  vertex operator algebras and lattice vertex operator algebras in terms of the affine varieties $\on{Sc}(V,\omega)$.

\end{document}